\documentclass[12pt]{amsart}

\usepackage{amsmath}
\usepackage{amsthm}
\usepackage{verbatim}
\usepackage{stmaryrd}
\usepackage{amsfonts,mathrsfs,amssymb}%,txfonts}
\usepackage{times}
\usepackage{fullpage}
\title{Non-uniformly flat affine algebraic hypersurfaces}
\author{Vamsi Pritham Pingali}
\email{vamsipingali@iisc.ac.in}
\address{Department of Mathematics \newline \indent Indian Institute of Science \newline \indent
Bangalore, India - 560012}
\author{Dror Varolin} 
\email{dror@math.stonybrook.edu}
\address{Department of Mathematics \newline \indent Stony Brook University \newline \indent Stony Brook, NY 11794-3651}

\newcommand{\noi}{\noindent}

\newcommand{\cO}{{\mathcal O}}
\newcommand{\cP}{{\mathcal P}}

\newcommand{\sB}{{\mathscr B}}
\newcommand{\sC}{{\mathscr C}}

\newcommand{\sI}{{\mathscr I}}

\newcommand{\sR}{{\mathscr R}}

\newcommand{\fB}{{\mathfrak B}}
\newcommand{\fC}{{\mathfrak C}}

\newcommand{\fH}{{\mathfrak H}}
\newcommand{\fJ}{{\mathfrak J}}

\newcommand{\vp}{\varphi} 
\newcommand{\ve}{\varepsilon}

\newcommand{\bbB}{{\mathbb B}}
\newcommand{\bbC}{{\mathbb C}}
\newcommand{\bbD}{{\mathbb D}}

\newcommand{\bbP}{{\mathbb P}}

\newcommand{\bbR}{{\mathbb R}}

\newcommand{\red}{\hfill $\diamond$}

\newcommand{\di}{\partial}
\newcommand{\dbar}{\bar \partial}

\newcommand{\re}{{\rm Re\ }}
\newcommand{\im}{{\rm Im\ }}

\newcommand{\relcomp}{\Subset}

\newcommand{\ii}{\sqrt{-1}}

\newcommand{\emb}{\hookrightarrow}

\def\Xint#1{\mathchoice 
{\XXint\displaystyle\textstyle{#1}}
{\XXint\textstyle\scriptstyle{#1}} 
{\XXint\scriptstyle\scriptscriptstyle{#1}}
{\XXint\scriptscriptstyle\scriptscriptstyle{#1}}
\!\int} 
\def\XXint#1#2#3{{\setbox0=\hbox{$#1{#2#3}{\int}$} 
\vcenter{\hbox{$#2#3$}}\kern-.5\wd0}} 
 
\def\dashint{\Xint-} 

\usepackage{hyperref}

\begin{document}
%\setlength{\textwidth}{24cm}
%\setlength{\textheight}{30cm}
%\maketitle

\theoremstyle{plain}
\newtheorem{thm}{\sc Theorem}
\newtheorem*{s-thm}{\sc Theorem}
\newtheorem{lem}{\sc Lemma}[section]
\newtheorem{d-thm}[lem]{\sc Theorem}
\newtheorem{prop}[lem]{\sc Proposition}
\newtheorem{cor}[lem]{\sc Corollary}

\theoremstyle{definition}
\newtheorem{conj}[lem]{\sc Conjecture}
\newtheorem*{sconj}{\sc Conjecture}
\newtheorem{prob}[lem]{\sc Open Problem}
\newtheorem*{s-prob}{\sc Open Problem}
\newtheorem{defn}[lem]{\sc Definition}
\newtheorem*{s-defn}{\sc Definition}
\newtheorem{qn}[lem]{\sc Question}
\newtheorem{ex}[lem]{\sc Example}
\newtheorem{rmk}[lem]{\sc Remark}
\newtheorem*{s-rmk}{\sc Remark}
\newtheorem{rmks}[lem]{\sc Remarks}
\newtheorem*{ack}{\sc Acknowledgment}

%\begin{abstract}
%\end{abstract}

\maketitle

\setcounter{tocdepth}2

\vskip .3in

%\begin{center}{\sc Abstract}\end{center}

%\noi {\small }
\begin{abstract}
The relationship between interpolation and separation properties of hypersurfaces in Bargmann-Fock spaces over $\bbC ^n$ is not well-understood except for $n=1$.  We present four examples of smooth affine algebraic hypersurfaces that are not uniformly flat, and show that exactly two of them are interpolating.  
\end{abstract}

\section{Introduction}

A natural problem is to establish a geometric characterization of all analytic subsets of $\bbC ^n$ that are \emph{interpolating} for the Bargmann-Fock space.  Let us be more precise.

\begin{defn}
A weight function $\vp$ is said to be a \emph{Bargmann-Fock weight} if  
\begin{equation}\label{bf-weight-hyp}
m \omega _o \le dd^c \vp \le M \omega _o
\end{equation}
for some  positive constants $m$ and $M$.  In this case the space 
\[
\sB _n (\vp) := \cO (\bbC ^n) \cap L^2 (e^{-\vp} dV)
\]
is called a \emph{Bargmann-Fock space}.  The weight $\vp (z) = |z|^2$ is called the \emph{standard} Bargmann-Fock weight, and the corresponding Hilbert space, denoted here simply as $\sB_n$, is called the \emph{standard}, or \emph{classical}, Bargmann-Fock space.
\end{defn}

\noi Here and below we write $d^c = \frac{\ii}{2} (\dbar - \di)$, so that $dd^c = \ii \di \dbar$, and we denote by 
\[
\omega _o = dd^c|z|^2
\]
the K\"ahler form for the Euclidean metric in $\bbC ^n$ with respect to the standard coordinate system.

Interpolation may be described as follows.  Let $(X,\omega)$ be a Stein K\"ahler manifold of complex dimension $n$, equipped with a holomorphic line bundle $L \to X$ with smooth Hermitian metric $e^{-\vp}$, and let $Z \subset X$ be a complex analytic subvariety of pure dimension $d$.  To these data assign the Hilbert spaces 
\[
\sB _n (X, \vp) := \left \{ F \in H^0(X,\cO _X(L))\ ;\ ||F||^2_X := \int _{X} |F|^2 e^{-\vp} \frac{\omega ^n}{n!} < +\infty \right \}
\]
and 
\[
\fB _d (Z, \vp) := \left \{ f \in H^0(Z,\cO _Z(L))\ ;\ ||f||^2_Z := \int _{Z_{\rm reg}} |f|^2 e^{-\vp} \frac{\omega ^d}{d!} < +\infty \right \}.
\]
Such Hilbert spaces are called \emph{(generalized) Bergman spaces}.  When the underlying manifold is $\bbC^n$ and the weight $\vp$ is Bargmann-Fock, we recover the Bargmann-Fock spaces just mentioned.  

We say that $Z$ is interpolating if the restriction map 
\[
\sR _Z : H^0(X, \cO _X (L)) \to H^0(Z, \cO _Z(L))
\]
induces a surjective map on Hilbert spaces.   (One can also ask whether the induced map is bounded, or injective, or has closed image, etc.)  If the induced map 
\[
\sR _Z : \sB _n (X, \vp) \to \fB _d (Z, \vp)
\]
is surjective then one says that $Z$ is an \emph{interpolation subvariety}, or simply \emph{interpolating}.

If $n \ge 2$ then even in the most elementary case $X = \bbC ^n$, $\omega =  \omega _o = dd^c |z|^2$ and $\vp (z) = |z|^2$ relatively little is known about which subvarieties (and even smooth manifolds) are interpolating.  (By way of contrast the case $n=1$ and $X= \bbC$ is rather well-understood; c.f. Section \ref{review-of-interp}.)

The present article focuses on the latter setting, and even more selectively, on the rather restricted class of smooth affine algebraic hypersurfaces.  The basic problem considered in this article is the following.

\medskip

\noi {\bf Basic Question:}  \emph{What geometric properties characterize  interpolating algebraic hypersurfaces for Bargmann-Fock spaces?}

\medskip

There are sufficient conditions on a hypersurface $Z$ so that it is interpolating for a Bargmann-Fock space.  For example, one has the following theorem, that generalizes a result in \cite{osv} about smooth surfaces to the possibly singular case.

\begin{d-thm}\textup{\cite{pv}}\label{interp-thm-pv-osv}
Let $\vp \in \sC ^2 (\bbC ^n)\cap {\rm PSH}(\bbC ^n)$.  (For example, $\vp$ can be Bargmann-Fock weight, i.e., satisfying \eqref{bf-weight-hyp}.)  Then every uniformly flat hypersurface $Z\subset \bbC ^n$ whose asymptotic upper density $D^+_{\vp}(Z)$ is less than $1$ is an interpolation hypersurface. 
\end{d-thm}

We shall recall the definition of the asymptotic upper density in Section \ref{review-of-interp}, in which we will provide a brief and biased overview of interpolation theory.  As for uniform flatness, a smooth hypersurface $Z \subset \bbC ^n$ is uniformly flat if there is a constant $\ve > 0$ such that the set 
\[
N_{\ve}(Z) = \left \{ x\in \bbC ^n\ ;\ B_{\ve} (x) \cap Z \neq \emptyset \right \}
\]
of all points of $\bbC ^n$ that are a distance less than $\ve$ from $Z$ is a tubular neighborhood of $Z$.  Equivalently, for any pair of distinct points $p, q\in Z$, if $D_p$ and $D_q$ denote the Euclidean complex disks of radius $\ve$ and centers $p$ and $q$ respectively such that $D_p \perp T_{Z,p}$ and $D_q \perp T_{Z,q}$ , then
\[
D_p \cap D_q = \emptyset.
\]
\begin{s-rmk}
The notion of uniform flatness was introduced in \cite{osv}, and extended to singular hypersurfaces in \cite{pv}, but since we will not use the latter here, we will not recall the definition in the singular case.
\red
\end{s-rmk}

In the case of a smooth hypersurface, Theorem \ref{interp-thm-pv-osv} has a very simple proof which we discovered in \cite{pv}.  We shall recall this proof in the last paragraph of Section \ref{review-of-interp}, after providing a brief and biased overview of the theory of interpolation, and stating a version of the $L^2$ extension theorem (Theorem \ref{ot-tak}).

The connection between uniform flatness of a hypersurface $Z$ and the surjectivity of $\sR_{Z}$ was shown in \cite{pv} to be somewhat more mysterious than previously believed.  

\begin{enumerate}
\item[(a)]  There is a holomorphic embedding $C$ of $\bbC$ in $\bbC ^2$ whose asymptotic upper density is zero, such that $C$ is not uniformly flat but nevertheless it is an interpolation hypersurface.   

\item[(b)] While uniform flatness is not necessary, it cannot be dropped completely; simple examples from the $1$-dimensional setting can be extended via cartesian product to give examples in dimension $2$ or more.
\end{enumerate}

In part to focus more on the role (or lack of role) of uniform flatness, but also for other reasons, it is interesting to restrict oneself to the class of algebraic hypersurfaces.  Indeed, as we shall recall in Section \ref{review-of-interp}, every algebraic hypersurface has zero asymptotic upper density.  (In this regard, the examples produced in \cite{pv} to demonstrate (a) and (b) are not algebraic.)

We have not yet succeeded in answering our basic question of  characterizing interpolating affine algebraic hypersurfaces.  However, the results we obtained, using techniques that are important and interesting in their own right, provide some data for the problem that we believe will be useful in attacking the basic question.

\subsection{Results}

Consider the smooth complex curves 
\[
C_1 := \{ (x,y) \in \bbC ^2\ ;\ x^2y^2 =1\} \quad \text{and} \quad C_2 := \{ (x,y) \in \bbC ^2\ ;\ xy^2 =1\}.
\]
The curve $C_1$ is an embedding of two disjoint copies of $\bbC^*$ embedded via the maps 
\[%begin{equation}\label{C1-emb}
\Psi _{1\pm}(t) := (t^{-1}, \pm t),
\]%end{equation}
and each of the components $C_{1\pm} = \Psi_{1\pm}(\bbC ^*)$ is uniformly flat.   The curve $C_2$ is a copy of $\bbC ^*$ embedded in $\bbC ^2$ via the map
\[%begin{equation}\label{C2-emb}
\Psi _2 (t) := (t^{-2}, t).
\]%end{equation}
Both $C_1$ and $C_2$ are not uniformly flat:
\begin{enumerate}
\item[($C_1$)]  Let $\delta > 0$.  The points $p_{\pm} := (\delta^{-1} , \pm \delta)$ both lie on $C_1$, and the disks perpendicular to $C_1$ at $p_{\pm}$ intersect at the point $I = (\delta ^{-1} - \delta ^3, 0)$; the distance from $I$ tp $p_{\pm}$ is $\delta(1+\delta ^4/4)^{1/2}$, which can be made smaller than any positive number by taking $\delta$ sufficiently small.

\item[($C_2$)] Again let $\delta > 0$.  The points $p_{\pm} = (\delta ^{-2}, \pm \delta)$ both lie on $C_2$, and the disks perpendicular to $C_2$ at $p_{\pm}$ intersect at the point $I = (\delta ^{-1} - \delta ^4/2, 0)$; the distance from $I$ to $p_{\pm}$ is computed to be $\delta(1+\delta ^6/4)^{1/2}$, which can be made smaller than any positive number by taking $\delta$ sufficiently small.
\end{enumerate}

\noi The first two results we state are the following theorems.
\begin{thm}\label{main-C1}
Let $\vp \in \sC ^2 (\bbC^2)$ satisfy \eqref{bf-weight-hyp}.  Then there exists $f \in \fB _1 (C_1, \vp)$ such that any holomorphic extension $F$ of $f$ to $\bbC ^2$ does not lie in $\sB _2(\vp)$.
\end{thm}

\begin{thm}\label{main-C2}
Let $\vp \in \sC ^2 (\bbC^2)$ satisfy \eqref{bf-weight-hyp}.  Then the restriction map $\sR_{C_2}: \sB_2 (\vp) \to  \fB _1 (C_2, \vp)$ is bounded and surjective.
\end{thm}

Next consider the smooth complex surfaces 
\[
S = \{ (x,y,z) \in \bbC ^3\ ;\ z= xy^2\} \quad \text{and} \quad \Sigma  = \{ (x,y,z) \in \bbC ^3\ ;\ z= x^2y^2\}.
\]
The surface $S$ and $\Sigma$ are both graphs over $\bbC ^2$, and hence embed in $\bbC ^3$ by the maps 
\begin{equation}\label{S-graph-emb}
\Phi (s,t) := (s,t,st^2) \quad \text{and} \quad \Psi (s,t) = (s,t,s^2t^2)
\end{equation}
respectively.  Unlike the curve $C_1$, both $S$ and $\Sigma$ are connected.   

Like the curves $C_1$ and $C_2$, the surfaces $S$ and $\Sigma$ are also not uniformly flat.  Heuristically speaking, by intersecting with the planes $x =c$ and letting $c \to \infty$, one obtains more and more eccentric parabolas.  More precisely, for $0<\delta <1$ the disks perpendicular to $S$ at the points $p_{\pm} := (\delta ^{-1},\pm \delta, \delta)$ intersect at the point $I:= (\delta ^{-1}-\frac{\delta^3}{2}, 0, \delta + \frac{\delta}{2})$, and 
\[
|I -p_{\pm}| = |(\delta ^3/2, \pm \delta , -\delta/2)| < 2 \delta. 
\]
Thus the neighborhood $N_{\ve}(S)$ is not a tubular neighborhood for any constant $\ve > 0$.  Similar considerations apply to $\Sigma$.

\begin{thm}\label{main-S}
Let $\vp \in \sC ^2 (\bbC^3)$ satisfy \eqref{bf-weight-hyp}.  Then the restriction map $\sR _S : \sB _3(\vp) \to \fB_2(S,\vp)$ is bounded and surjective.  
\end{thm}

\begin{thm}\label{main-sigma}
Let $\vp \in \sC ^2 (\bbC^3)$ satisfy \eqref{bf-weight-hyp}.  Then there exists $f \in \fB _2 (\Sigma, \vp)$ such that any holomorphic extension $F$ of $f$ to $\bbC ^3$ does not lie in $\sB _3(\vp)$.
\end{thm}

One might wonder what feature of the curve $C_2$ makes it interpolating, while $C_1$ is not interpolating.  In the case of $C_1$, two points with intersecting small orthogonal disks are always infinitely far apart in $C_1$, in the sense that they cannot be connected by a path.  On the other hand, the curve $C_2$ is more confusing:  the points $(\delta ^{-2}, \pm \delta)$ are rather far apart on $C_2$, but as it turns out, not far enough apart. 

As for the surface $S$, any two points with intersecting small orthogonal disks are always very close together in $S$, and moreover all such points are confined to a small neighborhood of the line $\{y=z=0\}$, which is a uniformly flat complex analytic submanifold of $\bbC^3$.  It is this feature of $S$ that makes it manageable. 

\begin{rmk}\label{not-flat-affine-graphs}
In \cite[p.87]{osv} a claim was made that ``\emph{it is not hard to see that}" the graph in $\bbC ^n$ of any polynomial in $n-1$ variables is uniformly flat.  Obviously $S$ is a counterexample to this claim when $n \ge 3$.  That being said, curves in $\bbC ^2$ that are graphs of polynomials in one complex variable are uniformly flat. (The curve $C_2$ shows that this is not the case for graphs of rational functions.)
\red
\end{rmk}

\subsection{Path to enlightenment} \label{light-ss}

Our struggles with Theorems \ref{main-C1}-\ref{main-sigma} compel us to tell story of our trajectory in establishing their proofs.

The surface $S$ was the first that we considered, and it came up precisely in the context of Remark \ref{not-flat-affine-graphs}.  Our initial expectation was that $S$ would not be interpolating, but we had difficulty writing down a proof.  In the meantime, since the surface $S$ was so hard to understand, we considered the curve $C_2$ in the hopes that it would provide a more manageable example.  We tried in several ways to prove that $C_2$ was not interpolating, not knowing it was impossible to do so.  Eventually we simplified things even further to the curve $C_1$, and finally we were successful in showing that $C_1$ was not interpolating.  

Eventually we realized that $S$ was indeed interpolating, and that the reason had to do with the fact that the non-flat regions were concentrated near the line $\{y=z=0\}$, which is a small, and interpolating, subset of $\bbC ^3$.  This was the key to the proof of Theorem \ref{main-S}.

Yet even after knowing that $S$ is interpolating, we continued to try to show that $C_2$ is not interpolating.  The rationale was that, in the plane, if two points on an algebraic curve are very close in the ambient space, they must be very far apart with respect to the distance induced by the Euclidean metric on the curve.  And this is indeed the case for $C_2$.  (We will explain later why having points that are far apart in the curve but close together in the ambient space could lead to a contradiction to interpolation.)  As it turned out, the pairs of non-flat points were not far enough apart for our approach to work.  So we started to wonder if perhaps $C_2$ was interpolating after all.  With this psychological shift, things changed quickly.

It occurred to us that $C_2$ can be seen as a uniformly flat subset of $S$, since it is cut out from $S$ by the plane $\{z=1\}$.  We conjectured that perhaps data from $C_2$ could be extended to $S$.  This turned out to be the case, and from there on it was clear how to extend data from $C_2$ to $\bbC ^2$:  extend the data to $S$, then extend the data on $S$ to $\bbC ^3$, and finally restrict to $\bbC ^2 \times \{1\}$.

After seeing that, unlike $C_2$, $C_1$ is not connected, we wondered if every smooth connected affine algebraic hypersurface is an interpolation hypersurface.  But by this point we had gained enough experience so as not to be easily led astray.  We realized that, like $C_2$ inside $S$, the curve $C_1$ is uniformly flat inside $\Sigma$.  If $\Sigma$ were interpolating, then we could extend data from $C_1$ to $\Sigma$ and then from $\Sigma$ to $\bbC ^3$, after which restriction to the plane $\bbC ^2 \times \{1\}$ yields a contradiction to Theorem \ref{main-C1}.

Of course, all of these sketches are a little imprecise.  As the reader might agree, the details require considerable care.  

\subsection{More ideas behind the proofs}

After establishing the tools that are needed, the proof of Theorem \ref{main-C1} is presented first.  The idea is as follows.  One constructs a function in $\fB_2 (C_1, \vp)$ that is very large at the point $(\delta ^{-2}, \delta)$ and very small at the point $(\delta ^{-2}, -\delta)$. The function is built using H\"ormander's Theorem, but there is some subtlety regarding the curvature of the weights.  Thus, in addition to H\"ormander's Theorem, one makes use of a technique-- first introduced by Berndtsson and Ortega Cerd\`a \cite{quimbo}-- that is discussed in Section \ref{tech-tools}.  

The next result to be proof is Theorem \ref{main-S}.  For its proof, we exploit the $L^2$ Extension Theorem  (Theorem \ref{ot-tak}) to construct our extensions in two different open sets; one large open set where $S$ is uniformly flat, and another small open set where $S$ is not uniformly flat.  There is a difficulty in extending from the non-uniformly flat subset.  This difficulty is overcome by a reduction to extending functions that vanish along the set where uniform flatness is violated.  Finally, H\"ormander's Theorem is used to patch together these two extensions. 

Theorem \ref{main-C2} is deduced from Theorem \ref{main-S} by again exploiting the $L^2$ Extension Theorem.  As we already mentioned, a simple but important observation is that the curve $C_2$ is the intersection of $S$ with the plane $\{z=1\}$ in $\bbC ^3$.  If we can extend data from $C_2$ to $S$, then by Theorem \ref{main-S} we can extend the data to $\bbC ^3$, and then restrict it to $\bbC^2 \times \{1\} \cong \bbC ^2$.  Thus the difficulty is to extend from $C_2$ to $S$.  The key feature is that since the plane $\bbC ^2 \times \{1\}$ is (uniformly) flat in $\bbC^3$, one suspects that $C_2$ is uniformly flat when viewed from within $S$. 

Perhaps it should be noted that the most difficult part of proving Theorem \ref{main-C2} is guessing that it, rather than its converse, is true.  The points violating uniform flatness, i.e., $(\delta ^{-2}, \pm \delta)$, are rather far apart in $C_2$ (with respect to the Riemannian distance induced by the Euclidean metric on the surface) but rather close in the ambient space.  Therefore any interpolation problem from this pair of points into the curve $C_2$ can be solved, which means that one can find a function that is very large at $(\delta ^{-2},\delta)$ and vanishes at $(\delta ^{-2},- \delta)$.  The extension of such a function would have very large $L^2$ norm, since its gradient would be huge.  However, in order to have good control over the norm of the extension, one needs a lot of curvature from the curve, and the Bargmann-Fock condition \eqref{bf-weight-hyp} turned out simply be too much for a connected algebraic curve.

We feel confident enough to make the following conjecture.

\begin{conj}
A smooth connected affine algebraic curve in $\bbC ^2$ is interpolating for any Bargmann-Fock space.
\end{conj}

Finally, Theorem \ref{main-sigma} is deduced from Theorem \ref{main-C2} in a manner that is the mirror image of the deduction of Theorem \ref{main-C1} from Theorem \ref{main-S}.  One shows that the curve $C_2$, obtained from $\Sigma$ by intersection with the plane $\bbC ^2 \times \{1\}$, is interpolating for $\Sigma$.  If $\Sigma$ were interpolating for $\bbC ^3$ then the data from $C_2$ could be extended first to $\Sigma$ and then to $\bbC^3$, and then it could be restricted to $\bbC ^2 \times \{1\} \cong \bbC ^2$.  The result would contradict Theorem \ref{main-C2}.

\medskip 

\begin{ack}
The first named author is partially supported by  the Young Investigator Award and by grant F.510/25/CAS-II/2018(SAP-I) from UGC (Govt. of India).  The second named author is grateful to Takeo Ohsawa for his interest in the article and for his helpful advice.
\end{ack}

\section{Background on asymptotic density, uniform flatness, and interpolation}\label{review-of-interp}

The theory of interpolation from complex analytic hypersurfaces in $\bbC^n$ began its development in the early 1990s with the work of Kristian Seip and several other collaborators.  Seip considered the problem of interpolation and sampling from $0$-dimensional analytic subvarieties in $\bbC$, giving a negative answer to the following question that arose in solid state physics:  
\begin{center}
{\it Is there a lattice $\Lambda$ in $\bbC$ such that restriction map $\sR_{\Lambda} : \sB_1 \to \fB_o (\Lambda)$ is a bijection?}
\end{center}
\label{seip-work}
\noi Seip showed that in fact there is no closed discrete subset of $\bbC$ for which the restriction map is a bijection.  To prove this non-existence, Seip defined an adaptation, in the Bargmann-Fock space, of a notion of asymptotic upper and lower densities introduced by Beurling for Hardy spaces.  The definition of the upper density and lower density of a closed discrete subset $\Gamma$ is 
\[
D^+(\Gamma) = \limsup _{r \to \infty} \sup _{z \in \bbC} \frac{\# \Gamma \cap D_r (z)}{r^2} \quad \text{and} \quad D^-(\Gamma) = \liminf _{r \to \infty} \inf _{z \in \bbC} \frac{\# \Gamma \cap D_r (z)}{r^2}
\]
respectively.  Clearly the upper density of $\Gamma$ is always larger than the lower density of $\Gamma$.  Seip showed that if $\sR_{\Gamma}$ is injective then $D^- (\Gamma) > 1$ and that if $\sR_{\Gamma}$ is surjective then $D^+(\Gamma) < 1$, thus obtaining the negative answer to the above question.  Seip also showed that if $\sR _{\Gamma}$ is surjective then $\Gamma$ is uniformly separated in the Euclidean distance in $\bbC$, and that if $\sR _{\Gamma}$ is injective with closed range then $\Gamma$ is a finite union of uniformly separated sequences.  Conversely, Seip and Wallsten showed that if $\Gamma$ is uniformly separated and $D^+(\Gamma) < 1$ then $\sR_{\Gamma}$ is surjective, while if $\Gamma$ is a finite union of uniformly separated sequences $\Gamma _1,...,\Gamma _N$, such that $D^-(\Gamma_i) > 1$ for some $i$ then $\sR_{\Gamma}$ is injective with closed range.  Thus a rather complete picture is obtained: see \cite{seip-C,sw}.

The results of Seip and Wallst\'en for the standard Bargmann-Fock space were extended to general Bargmann-Fock spaces on $\bbC$ by Berndtsson and Ortega Cerd\`a (sufficiency) \cite{quimbo} and by Ortega Cerd\`a and Seip (necessity) \cite{quimseep}.  Other domains besides $\bbC$ have been considered, c.f. for example \cite{seip-disk,sv-rs,oc-rs,v-rs-flat,v-rs-2}, but the present article focuses on the Bargmann-Fock situation.  

Let $Z \subset \bbC^n$ be an analytic hypersurface.  For any such hypersurface there exist functions $T \in \cO (\bbC^n)$ such that $dT(p)\neq 0$ for at least one $p$ in every connected component of $Z$.  Such a function $T$ will be called a \emph{defining function} for $Z$.  Any two defining functions $T_1$ and $T_2$ for $Z$ are related by $T_2 = e^{f} T_1$ for some $f \in \cO (\bbC ^n)$.  

Given a defining function $T$ for $Z$, for each $r > 0$ we can define the function 
\[
\lambda ^T _r (z) := \frac{n!}{(\pi r^2)^n}\int _{B_r (z)} \log |T|^2 dV = \dashint _{B_r (z)} \log |T|^2 dV .
\]
Observe that if $\tilde T = e^f T$ is another defining function for $Z$ then 
\[
\lambda ^{\tilde T} _r  = 2 \re f  + \lambda ^T_r.
\]
It follows that the functions 
\[
\sigma_r ^Z := \log |T|^2 - \lambda ^T_r : \bbC ^n \to \bbR \cup \{-\infty\} \quad \text{and} \quad S^Z_r := |dT|^2 e^{-\lambda ^T_r} : Z \to [0, \infty),
\]
called the \emph{singularity} and the \emph{separation function} of $Z$, are independent of the defining function $T$, as is the locally bounded $(1,1)$-current 
\[
\Upsilon ^Z _r := dd^c \lambda ^T _r = \frac{n!}{(\pi r^2)^n} \mathbf{1}_{B_r(0)} * [Z],
\]
called the \emph{mass tensor} of $Z$.

Note that the mass tensor is a non-negative Hermitian $(1,1)$-form.  Its size is therefore governed by its trace $||\Upsilon^Z_r(z)||$ given by 
\[
||\Upsilon^Z_r(z)|| \omega_o ^n := n \Upsilon ^Z_r \wedge \omega _o ^{n-1} = \frac{\omega_o^n}{(n-1)!}\dashint _{B_r (z)} [Z] \wedge \omega _o ^{n-1}
\]
which is the ratio of the area of $Z \cap B_r (z)$ to the volume of $B_r(z)$.  It follows that for any $v \in \bbC ^n$ 
\begin{equation}\label{estimate-for-mass-tensor}
\Upsilon ^Z_r(z)(v,\bar v) \le \frac{\omega _o (v,\bar v)}{(n-1)!} \int _{\bbC ^n} [Z] \wedge \omega _o ^{n-1}.
\end{equation}

Next let $\vp \in \sC ^2 (\bbC ^n)$ be a Bargmann-Fock weight, i.e., a weight satisfying the bounds \eqref{bf-weight-hyp} on its curvature.  One can form the \emph{mean weight} $\vp _r \in \sC ^2 (\bbC ^2)$ defined by 
\[
\vp _r (z) := \dashint _{B_r (z)} \vp dV.
\]
As we will see below (Lemma \ref{w12-for-ddbar}) the weight $\vp$ can be written in $B_r (z)$ in the form 
\[
\vp(\zeta) = m ||\zeta - z||^2 + 2 \re g(\zeta) + \psi(\zeta)
\]
for some $g \in \cO (B_r(z))$, and some $\psi \in \sC^2(B_r (z))$ whose $\sC^1$-norm is bounded independent of $z$.  (In Lemma \ref{w12-for-ddbar} we have $m ||\zeta||^2$ rather than $m ||\zeta - z||^2$, but the proof is the same.)  It follows from Taylor's Theorem that
\[
\left | \vp (z) - \vp_r (z) \right | = \left | \dashint _{B_r (0)} (\vp (\zeta +z) - \vp (z) ) dV (\zeta) \right | \le C_r
\]
for some constant $C_r$ that is independent of $z$.  Thus the Hilbert spaces $\sB_n (\vp)$ and $\sB _n (\vp _r)$ are quasi-isometric, as are the Hilbert spaces $\fB_{n-1} (Z,\vp)$ and $\fB_{n-1} (Z,\vp_r)$.

\subsection{Asymptotic Density}

We can now generalize the notion of asymptotic upper and lower densities as follows.

\begin{defn}\label{density-defn}
Let $Z \subset \bbC ^n$ be a possibly singular analytic hypersurface.  
\begin{enumerate}
\item[(a)]  The asymptotic upper density of $Z$ with respect to the Bargmann-Fock weight $\vp$ is 
\[
D^+_{\vp} (Z) := \limsup _{r \to \infty} \sup _{z \in \bbC ^n} \sup _{v \in \bbC ^n - 0} \frac{\int _{B_r(z)}dd^c \log |T|^2 (v,\bar v) dV}{\int _{B_r(z)}dd^c\vp (v,\bar v)dV} =  \limsup _{r \to \infty} \sup _{z \in \bbC ^n} \sup _{v \in \bbC ^n - 0} \frac{\Upsilon^Z_r (v,\bar v)}{dd^c \vp_r(z) (v,\bar v)}
\]
\item[(b)] The asymptotic lower density of $Z$ with respect to the Bargmann-Fock weight $\vp$ is 
\[
D^-_{\vp} (Z) := \liminf _{r \to \infty} \inf _{z \in \bbC ^n} \inf _{v \in \bbC ^n - 0} \frac{\int _{B_r(z)} dd^c \log |T|^2 (v,\bar v) dV}{\int _{B_r(z)} dd^c \vp (v,\bar v)dV} =  \liminf _{r \to \infty} \inf _{z \in \bbC ^n} \inf _{v \in \bbC ^n - 0} \frac{\Upsilon^Z_r (v,\bar v)}{dd^c \vp_r(z) (v,\bar v)}
\]
\end{enumerate}
\end{defn}
In other words, the upper density $D^+_{\vp}(Z)$ is the infimum of all positive numbers $a$ such that 
\[
dd^c \vp_r - \frac{1}{a} \Upsilon^Z_r > 0,
\]
while the lower density $D^-_{\vp} (Z)$ is the supremum of all numbers $c$ such that there exists $z, v \in \bbC ^n$ satisfying 
\[
dd^c \vp _r (z)(v,\bar v) - \frac{1}{c} \Upsilon^Z_r(z) (v,\bar v) < 0.
\]
Note that 
\[
D^-_{\vp} (Z) \le D^+_{\vp} (Z)
\]
and that either of the densities can be infinite.

In the present article, the following simple proposition is relevant.

\begin{prop}\label{alg-zero-density}
If $Z$ is an algebraic hypersurface in $\bbC ^n$ then $D^+_{\vp}(Z) = 0$.
\end{prop}

\begin{proof}
Since an algebraic hypersurface $Z$ of degree $d$ is locally a $d$-sheeted cover of a complex hyperplane, the area of $Z\cap B_r(z)$ is
\[
\int _{B_r(z)} [Z]\wedge \frac{\omega _o^{n-1}}{(n-1)!} = O(r^{2n-2})
\]
uniformly in $z$.  Since $dd^c \vp _r (v,\bar v) \ge m r^{2n} \omega _o (v,\bar v)$ the result follows from \eqref{estimate-for-mass-tensor}.
\end{proof}

\subsection{Uniform flatness}

As we already recalled in the introduction, a smooth hypersurface $Z \subset \bbC ^n$ is uniformly flat if there is a positive constant $\ve > 0$ such that 
\[
N_{\ve}(Z) = \left \{ x\in \bbC ^n\ ;\ B_{\ve} (x) \cap Z \neq \emptyset \right \}
\]
(the $\ve$-neighborhood of $Z$) is a tubular neighborhood of $Z$.  In our previous article \cite{pv} we established the following result.

\begin{prop}\label{flat-and-dt}\cite[Lemma 4.11]{pv}
If a smooth hypersurface $Z \subset \bbC ^n$ is uniformly flat for each $r > 0$ the separation function 
\[
S ^Z _r := |dT|^2e^{-\lambda ^T_r} :Z \to \bbR _+
\]
is bounded below by a positive constant $C_r$.
\end{prop}

In dimension $n=1$ the converse of Proposition \ref{flat-and-dt} is true as well.  And although we suspect it is the case, we don't know if the converse is also true in higher dimensions.

\subsection{Interpolation and sampling}

\begin{defn}
Let $Z$ be a pure $k$-dimensional complex subvariety of $\bbC ^n$, and suppose $\bbC ^n$ is equipped with a Bargmann-Fock weight $\vp \in \sC ^2 (\bbC^n)$.
\begin{enumerate}
\item[(I)] We say that $Z$ is an interpolation subvariety if the restriction map 
\[
\sR _Z : \cO (\bbC ^n) \to \cO (Z)
\]
induces a well-defined and surjective map $\sR _Z : \sB _n (\vp) \to \fB _{k} (Z,\vp)$.
\item[(S)] We say that $Z$ is a sampling subvariety if the restriction map 
\[
\sR _Z : \cO (\bbC ^n) \to \cO (Z)
\]
induces a well-defined and injective map $\sR _Z : \sB _n (\vp) \to \fB _{k} (Z,\vp)$ whose image is closed.
\end{enumerate}
\end{defn}

In connection with interpolation, we have already mentioned Theorem \ref{interp-thm-pv-osv} for hypersurfaces.  For $k < n-1$ very little is known about interpolation.  The most interesting case is $k=0$, which would be most useful in applications. (There are some partial results in \cite{osv}, but these results are not decisive.)  It is known to experts that if $k=0$ and $n > 1$ then it is certainly not density that governs whether or not a sequence of points is interpolating (or for that matter, sampling).  Nevertheless, the density does have to be somewhat constrained.  An interesting necessary condition was introduced in \cite{lindholm}, and very recently improved in \cite{quim-et-al}.

The following simple proposition is very useful.

\begin{prop}[Bounded interpolation operators] \label{surj-has-bdd-inverse}
Let $Z \subset \bbC ^n$ be a complex subvariety and let $\vp \in \sC ^2 (\bbC ^n)$.  If the restriction 
\[
\sR _Z : \sB _n (\vp) \to \fB _{n-1} (Z, \vp)
\]
is surjective then there is a bounded section $I : \fB_{n-1}(Z, \vp) \to \sB _n (\vp)$ of $\sR_Z$.
\end{prop}

\begin{proof}
We define 
\[
I : \fB _{n-1} (\vp, Z) \to \sB _n (\vp)
\]
by letting $I(f)$ be the extension of $f$ having minimal norm in $\sB_n(\vp)$.  Equivalently, if we let $\sI _Z$ denote the sheaf of germs of holomorphic functions vanishing on $Z$, and write 
\[
\fJ_Z (\vp) := H^0(\bbC ^n, \sI _Z) \cap \sB _n (\vp),
\]
then $I(f)$ is the unique extension of $f$ to $\sB _n (\vp)$ such that 
\[
\int _{\bbC} I(f) \overline{G} e^{-\vp} dV = 0 \quad \text{for all }G \in \fJ_Z (\vp).
\]

By the Closed Graph Theorem the section $I : \fB _{n-1} (Z, \vp) \to \sB _n (\vp)$ is bounded if it has closed graph.  To show the latter, let $f _j \to f$ in $\fB_{n-1} (Z, \vp)$ and let $I(f_j) \to F$ in $\sB _n (\vp)$.  Then for each $G\in \fJ _Z$ one has 
\[
\int _{\bbC ^n} F \overline{G} e^{-\vp} dV = \lim _{j \to \infty} \int _{\bbC ^n} I(f_j) \overline{G} e^{-\vp} dV = 0.
\]
By the weighted Bergman inequality (proposition \ref{w-Bergman-ineq}) the $L^2$ norm controls the $L^{\infty}_{\ell oc}$ norm for holomorphic functions, and hence by Montel's Theorem the two limits are, perhaps after passing to subsequences, locally uniform.  It follows immediately that $F$ is an extension of $f$.  Hence $F = I(f)$, and the proof is complete.
\end{proof}

We end this subsection by noting that a subvariety is a sampling set if and only if 
\[
\int _{\bbC ^n} |F|^2 e^{-\vp} \omega _o ^n \lesssim \int _{Z_{\rm reg}} |F|^2 e^{-\vp} \omega _o ^k \lesssim \int _{\bbC ^n} |F|^2 e^{-\vp} \omega _o ^n
\]
holds for all $F \in \sB_n (\vp)$.  In the case of a smooth hypersurface, \cite{osv} establishes a companion result to Theorem \ref{interp-thm-pv-osv} the generalizes the positive direction of the sampling theorems established in generalized Bargmann-Fock spaces over $\bbC$ established by Berndtsson, Ortega Cerd\`a and Seip.  Surely there is also an analogue for singular, uniformly flat hypersurfaces, but the details have not been worked out.

\subsection{$\mathbf{L^2}$ extension after Ohsawa and Takegoshi}\label{OT-subsection}

Among the most sophisticated and useful set of results in complex analysis and geometry is the collection of theorems on $L^2$ extension that have come to be known as extension theorems of Ohsawa-Takegoshi type.  The name derives from the first fundamental result regarding $L^2$ extension in several complex variables, which was established by Ohsawa and Takegoshi in their celebrated article \cite{ot}.  Since that time, new proofs and extensions of the original result have been established by many authors, too numerous to state here.  The following version, established by the second author in \cite{dv-tak}, will be a convenient version for our purposes.

\begin{d-thm}\label{ot-tak}
Let $X$ be a Stein manifold with K\"ahler metric $\omega$, and let $Z \subset X$ be a smooth hypersurface.  Assume there exists a section $T \in H^0(X,\cO_X(L_Z))$ and a metric $e^{-\lambda}$ for the line bundle $L_Z \to X$ associated to the smooth divisor $Z$, such that $e^{-\lambda}|_Z$ is still a singular Hermitian metric, and 
\[
\sup _X |T|^2e^{-\lambda} \le 1.
\]
Let $H \to X$ be a holomorphic line bundle with singular Hermitian metric $e^{-\psi}$ such that $e^{-\psi}|_Z$ is still a singular Hermitian metric.  Assume there exists  $s \in (0,1]$ such that 
\begin{equation}\label{curv-ot}
\ii (\di \dbar \psi +{\rm Ricci}(\omega)) \ge(1+ ts) \ii \di \dbar \lambda _Z
\end{equation}
for all $t \in [0,1]$.  Then for any section $f \in H^0(Z,\cO_Z(H))$ satisfying 
\[
\int _Z \frac{|f|^2e^{-\psi}}{|dT|_{\omega}^2e^{-\lambda }}dA_{\omega} <+\infty 
\]
there exists a section $F\in H^0(X,\cO_X(H))$ such that 
\[
F|_Z=f \quad \text{and} \quad \int _X |F|^2e^{-\psi} dV_{\omega} \le \frac{24\pi}{s}\int _Z \frac{|f|^2e^{-\psi}}{|dT|_{\omega}^2e^{-\lambda }}dA_{\omega}.
\]
\end{d-thm}

$L^2$ extension theorems for higher codimension subvarieties also exist.  If the subvariety is cut out by a section of some vector bundle whose rank is equal to the codimension, with the section being  generically transverse to the zero section, then the result is very much analogous to Theorem \ref{ot-tak}.   For general submanifolds or subvarieties the result requires more normalization.  

The reader will notice that Theorem \ref{ot-tak} does not mention uniform flatness, and that density is not explicitly stated here.  However, the result does address both issues in a slightly more hidden way.  The issue of density is captured by the curvature conditions, while uniform flatness, or rather the absence of requiring uniform flatness, is dealt with by introducing the denominator $|dT|^2_{\omega}e^{-\lambda}$ in the norm on the hypersurface.  The following proof of Theorem \ref{interp-thm-pv-osv} provides a nice illustration.

\begin{proof}[Proof of Theorem \ref{interp-thm-pv-osv}]
In Theorem \ref{ot-tak} let $X= \bbC ^n$, $\psi= \vp$ and $\omega = \omega _o$.  Fix any $T\in \cO(\bbC ^n)$ whose zero locus is $Z$, such that $dT (z) \neq 0$ for all $z \in Z$.  Set 
\[
\lambda(z)  = \frac{1}{{\rm Vol} B_r(0)}\int _{B_r (0)} \log |T(z- \zeta)|^2 dV(\zeta).
\]
Choose $s \in (0,1]$ such that $D_{\vp} ^+(Z) < \frac{1}{1+s}$.  Then by definition \ref{density-defn} the curvature hypothesis \eqref{curv-ot} is satisfied, and thus we see that for any $f \in \cO (Z)$ such that 
\begin{equation}\label{ot-norm-on-hyp}
\int _Z \frac{|f|^2e^{-\vp}}{S^Z_r}dA_{\omega_o} < +\infty
\end{equation}
there exists $F \in \cO (\bbC ^n)$ such that 
\[
F|_Z = f \quad \text{and} \quad \int _{\bbC ^n} |F|^2 e^{-\vp} dA_{\omega _o} < +\infty.
\]
Since $Z$ is uniformly flat, Proposition \ref{flat-and-dt} implies that every $f \in \fB _{n-1} (Z, \vp)$ satisfies \eqref{ot-norm-on-hyp}, and thus Theorem \ref{interp-thm-pv-osv} is proved.
\end{proof}

\begin{rmk}
Note that something slightly stronger than Theorem \ref{interp-thm-pv-osv} is proved.  In fact, the bounded extension operator guaranteed by Proposition \ref{surj-has-bdd-inverse} is rather uniformly bounded.  Its norm is bounded by a constant that depends only on the density $D^+_{\vp}(Z)$ and on the separation constant 
\[
\sup \{ \ve > 0 \ ;\ U_{\ve} (Z) \text{ is a tubular neighborhood}\},
\]
or equivalently, the lower bound on the separation function.
\red
\end{rmk}

\section{The QuimBo trick}\label{tech-tools}

A basic principle in the study of generalized Bargmann-Fock spaces is that, locally, generalized Bargmann-Fock weights differ from standard Bargmann-Fock weights (i.e., weights that are quadratic polynomials and whose (therefore constant) curvature is strictly positive) by a harmonic function and a bounded term.   The basic result used to establish this decomposition is the following lemma, which is a minor generalization of a technique first introduced by Berndtsson and Ortega-Cerd\'a in dimension $1$ in \cite{quimbo}.  The technique has since affectionately come to be known as the \emph{QuimBo Trick}.

\begin{lem}\label{w12-for-ddbar}
There exists a constant $C> 0$ with the following property.  Let $\omega$ be a continuous closed $(1,1)$-form on a neighborhood of the closed unit polydisk $\overline{\bbD^k}$ in $\bbC ^k$, such that 
\[
- M \omega _o \le \omega \le M \omega _o
\]
for some positive constant $M$.  Then there exists a function $\psi \in \sC^2 (\bbD^k)$ such that 
\[
dd^c \psi = \omega \quad \text{and} \quad \sup _{\bbD^k} (|\psi| + |d\psi|) \le CM.
\]
\end{lem}

By scaling, one sees that in the polydisk of radius $(R,...,R)$ one has the same estimate with $M$ replaced by $MR^2$.  However, if the radii of the polydisk are not all the same, one can get a better estimate.

\subsection{Normalization of the weights}

\begin{lem}\label{weight-normalization}
There exists a constant $C> 0$ with the following property.  Let $\vp \in \sC^2(\bbC^k \times \bbC)$ satisfy 
\[
m dd^c |z|^2 \le dd^c \vp \le M dd^c |z|^2
\]
for some positive constants $m$ and $M$.  Then for each $r \in (0, 1]$ and each polydisk $D_R^k(0) \relcomp \bbC ^k$ with polyradius $R=(R_1,...,R_k) \in (0,\infty]^k$ and center $0$ there exists a plurisubharmonic function $\psi=\psi _{R,r} \in \sC^2(D^k_R(0)\times D_r(0))$ and a holomorphic function $g = g_{R,r}\in \cO (D^k_R(0) \times D_r(0))$ satisfying
\[
\vp = m |\cdot |^2 + \psi +2 \re g \quad \text{and} \quad \sup _{D^k_R(0)\times D_r(0)} |\psi|+|d\psi| \le  C\cdot (M-m)\left (r \log \tfrac{e}{r} + r\right ),
\]
where $C$ is a universal constant independent of the weight $\vp$, the radius $r$ and the polyradius $R$.
\end{lem}

\begin{proof}
Let $\omega=dd^c (\vp-m |\cdot |^2)$. Then $0\leq \omega \leq (M-m) \omega _o$. Suppose $\chi : \mathbb{R} \rightarrow \mathbb{R}_{\geq 0}$ is a smooth function equal to $1$ on $[-1,1]$ and $0$ outside $(-2,2)$. Now define $\psi : D_R^k(0)\times D_r (0) \rightarrow \mathbb{R} $ by the formula 
\[
\psi (z^1,...,z^{k+1}) = \frac{1}{\pi} \int _{|\zeta| < 2r}  \chi\left (\frac{\zeta}{r} \right )\omega _{k+1\overline{k+1}}(z^1, z^2,...,z^k, \zeta) \log |z^{k+1}-\zeta |^2 dA(\zeta),
\]
It is well-known that $dd^c \psi = \omega$ on $D_R^k (0)\times D_r(0)$.  It follows that the function $\vp - m |\cdot |^2 - \psi$ is pluriharmonic on the simply connected set $D_R(0)^k\times D_r(0)$, and thus equals $2 \re g$ for some $g \in \cO (D_R(0)^k\times D_r(0))$.  Hence in particular, $\psi \in \sC ^2(D_R^k(0)\times D_r(0)).$

The bound on $|\psi|$ follows from an obvious estimate.  As for the bound on $|d\psi|$, notice that 
\[
\frac{\di \psi }{\di z^{k+1}}(z,z^{k+1}) = \frac{1}{\pi} \int _{|\zeta| < 2r}   \chi\left (\frac{\zeta}{r} \right ) \omega _{k+1\overline{k+1}}(z^1, z^2,...,z^k, \zeta) \frac{dA(\zeta)}{z^{k+1}-\zeta },
\]
which is estimated using polar coordinates in $\zeta$ centered at $z^{k+1} \in D_r(0)$.  As for the other partial derivatives, for $1 \le j \le k$ we have 
\begin{eqnarray*}
\frac{\di \psi }{\di z^{j}}(z,z^{j}) &=& \frac{1}{\pi} \int _{|\zeta| < 2r}   \chi\left (\frac{\zeta}{r} \right ) \frac{\di}{\di z ^j} \omega _{k+1\overline{k+1}}(z^1, z^2,...,z^k, \zeta) \log |z^{k+1}-\zeta |^2 dA(\zeta)\\
&=& \frac{1}{\pi} \int _{|\zeta| < 2r}   \chi\left (\frac{\zeta}{r} \right ) \frac{\di}{\di \zeta} \omega _{j\overline{k+1}}(z^1, z^2,...,z^k, \zeta)\log |z^{k+1}-\zeta |^2 dA(\zeta)\\
&=& - \frac{1}{\pi} \int _{|\zeta| < 2r}   \chi\left (\frac{\zeta}{r} \right ) \omega _{j\overline{k+1}}(z^1, z^2,...,z^k, \zeta) \frac{dA(\zeta)}{z^{k+1}-\zeta } \nonumber \\
 &-&   \frac{1}{\pi} \int _{|\zeta| < 2r}   \frac{1}{r}\chi^{'}\left (\frac{\zeta}{r} \right ) \omega _{j\overline{k+1}}(z^1, z^2,...,z^k, \zeta) \log |z^{k+1}-\zeta |^2 dA(\zeta),
\end{eqnarray*}
where the second equality follows because $\di \omega = 0$ and the third equality is obtained via integration by parts. Since $\omega \leq (M-m) \omega_0$, $|\omega _{j \overline{k+1}}| \le M-m$, and the proof is complete. 
\end{proof}
of Lemma \ref{w12-for-ddbar}.

\begin{prop}\label{w-Bergman-ineq}
Let $\vp \in \sC ^2 (\bbC ^n)$ be a smooth weight function such that 
\begin{equation}\label{bounded-BF-curvature-condn}
- M \omega _o \le dd^c\vp \le M \omega _o
\end{equation}
for some positive constant $M$.  Then for each $r > 0$ there exists a constant $C_r$ depending on $r$ and $M$ such that if $F \in \sB_n (\vp)$ then for any $z \in \bbC ^n$
\begin{equation}\label{c0-bi}
(|F|^2 e^{-\vp})(z) \le C_r \int _{B^n_r(z)} |F|^2 e^{-\vp} dV
\end{equation}
and
\begin{equation}\label{c1-bi}
\left | d (|F|^2 e^{-\vp})\right |(z) \le C_r \int _{B^n_r(z)} |F|^2 e^{-\vp} dV
\end{equation}
\end{prop}

\begin{proof}
By rescaling and translating, we may assume that $r=1$ and $z=0$.  By Lemma \ref{w12-for-ddbar} applied to the form $\omega = dd^c (\vp - \vp (0)) = dd^c \vp$ there exists a function $\psi$ such that 
\[
dd^c \psi = dd^c \vp \quad \text{and} \quad \sup _{\bbB} |\psi| + |d \psi| \le C_o
\]
for some positive constant $C_o$.  It follows from the equation that $\psi -\psi (0) - \vp + \vp (0) = 2 \re G$ for some holomorphic function $G$ whose real part vanishes at $0$.  The imaginary part of $G$ can be chosen arbitrarily; for example we can take it to be $\int _0 ^z d^c (\psi - \vp)$, where the integral is over any curve in $\bbB$ originating at $0$ and terminating at $z$.  This choice yields the property $G(0) = 0$.  Thus we have 
\[
\sup _{\bbB} |\vp - \vp (0) + 2 \re G| + | d (\vp + 2\re G)| \le |\psi (0)| + \sup _{\bbB} |\psi| + |d \psi| \le 2 C_o.
\]
We therefore have 
\[
|F|^2e^{-\vp} = |Fe^G|^2 e^{-\vp (0)} e^{-\vp +\vp (0) - 2 \re G}
\]
Since the last factor is bounded, it can be eliminated from consideration, and we are reduced to the unweighted case (for the holomorphic function $Fe^G$).  The unweighted case is an elementary exercise in complex analysis (with a number of solutions), and is left to the reader.
\end{proof}

\begin{rmk}\label{proof-of-wbi-rmk}
Note that the proof of Proposition \ref{w-Bergman-ineq} yields a slightly more general fact:  if $\Omega \subset \bbC ^n$ is an open set and $F \in \cO (\Omega)$ satisfies 
\[
\int _{\Omega} |F|^2 e^{-\vp} dV < +\infty
\]
where $\vp \in \sC ^2 (\Omega)$ satisfies $-M \omega _o \le dd^c \vp \le M \omega _o$ \emph{only in $\Omega$}, then \eqref{c0-bi} and \eqref{c1-bi} hold for any $z \in \Omega$ and $r \in (0,\infty)$ such that $B_r (z) \subset \Omega$.
\red
\end{rmk}

\subsection{Interpolation sequences in $\mathbf{\bbC}$ are uniformly separated}\label{interp=>sep-1d}

In Section \ref{review-of-interp} (more precisely, in the first paragraph on Page \pageref{seip-work}) we noted that interpolation sequence in Bargmann-Fock spaces are uniformly separated. Let us recall the proof from \cite{quimseep}.

Let $\vp$ be a Bargmann-Fock weight on $\bbC$ and let $\Gamma \subset \bbC$ be a closed discrete subset such that  
\[
\sR_{\Gamma}: \sB_1(\vp) \to \fB_0(\Gamma) := \left \{ f : \Gamma \to \bbC \ ;\ \sum _{\gamma \in \Gamma} |f(\gamma)|^2 e^{-\vp(\gamma)} < +\infty \right \}
\]
is surjective.  Now choose $\gamma_o, \gamma_1 \in \Gamma$ distinct.  The function $f : \Gamma \ni \gamma \mapsto e^{\vp (\gamma_o)/2}\delta _{\gamma _o, \gamma}$ has norm 
\[
||f||^2 = |f(\gamma _o)|^2 e^{-\vp(\gamma _o)} = 1.
\]
By Proposition \ref{surj-has-bdd-inverse} there exists $F \in \sB _1(\vp)$ such that $||F||^2 \le C$ where $C$ is independent of $f$ (hence of $\gamma _o$).  It follows that 
\begin{eqnarray*}
\frac{1}{|\gamma _o - \gamma _1|} &=& \left | \frac{|f(\gamma _o)|^2 e^{-\vp(\gamma _o)} - |f(\gamma _1)|^2 e^{-\vp(\gamma _1)}}{|\gamma _o - \gamma _1|} \right |\\
&=& \left | \frac{1}{\gamma _1 - \gamma _o}\int _0 ^1 \frac{d}{dt} |f(\gamma _o +t (\gamma _1 - \gamma _o))|^2 e^{-\vp(\gamma _o+t (\gamma _1 - \gamma _o))}  dt\right |\\
& \le & \sup _{\bbC} \left | d (|f|^2 e^{-\vp}) \right |.
\end{eqnarray*}
By \eqref{c1-bi} of Proposition \ref{w-Bergman-ineq}, $|\gamma _o - \gamma _1| \gtrsim ||F||^{-2} \ge C^{-1}$, which is what we wanted to show.

%A corollary of \eqref{c1-bi} is the following proposition.

%\begin{prop}\label{slow-growth-in-B}
%Let $\vp$ be as in Proposition \ref{w-Bergman-ineq}.  For every $a>0$ there exists $\ve >0$ such that for each $F \in \sB _n(\vp)$ with $|F(z)|^2e^{-\vp(z)} \ge a$,  $||F|| \le 1$, each $z \in \bbC ^n$ and each $w \in B_{\ve}(z)$,
%\[
% |F(w)|^2e^{-\vp(w)} \ge \frac{1}{2}|F(z)|^2e^{-\vp(z)}.
% \]
%\end{prop}

%\begin{proof}
%If not, then there exists $a> 0$ and a sequence of points $z_j, w_j$ such that $|w_j - z_j| < \frac{1}{2j}$, $|F(z_j)|^2e^{-\vp(z_j)} \ge a$ and $|F(w_j)|^2e^{-\vp(w_j)} < |F(z_j)|^2e^{-\vp(z_j)}/2$.  But then 
%\begin{eqnarray*}
%ja & \le &\frac{1}{|w_j - z_j|} \left | |F(w_j)|^2e^{-\vp(w_j)} - |F(z_j)|^2e^{-\vp(z_j)} \right |\\
%&\le & \frac{1}{|w_j - z_j|}\left | \int _0 ^1 \frac{d}{dt} |F(w_j + t (z_j-w_j)|^2 e^{-\vp(w_j+t (z_j-w_j))} dt \right |\\
%&\le & \int _{0} ^1 \left | d  \left ( |F(w_j  + t (z_j -w_j)|^2 e^{-\vp(w_j+t (z_j-w_j))} \right ) \right | dt \\
%& \le & \sup _{\bbC ^n} \left | d(|F|^2 e^{-\vp})\right |,
%\end{eqnarray*}
%which contradicts \eqref{c1-bi} of Proposition \ref{w-Bergman-ineq}.
%\end{proof}

%\begin{rmk}
%The crucial point to emphasize in Proposition \ref{slow-growth-in-B} is that $\ve > 0$ depends only on $a$, and not on the points $z$ and $w$, or the function $F$.
%\red
%\end{rmk}
\section{Proof of Theorem \ref{main-C1}}

We begin by considering the case of the standard Bargmann-Fock space, and then extend the proof to the general case.  Even in the standard case we were not able to write down a simple, explicit  example of a function in $\fB_1 (C_1)$ that has no extension in $\sB _2$.  We require $L^2$ methods to construct our function.

\subsection{The standard Bargmann-Fock space}\label{bf-case-C1}

\subsubsection{\bf Reduction}  

The strategy of our proof consists in seeking a function $f \in \fB_1 (C_1)$ for which any holomorphic extension would violate \eqref{c1-bi} of Proposition \ref{w-Bergman-ineq}.  (See Paragraph \ref{interp=>sep-1d} for the case of sequences in $\bbC$.)  With this general goal in mind, let $T_1^{\pm} , T_1 \in \cO (\bbC ^2)$ be defined by  
\[
T^{\pm}_1 (x,y) = xy \mp 1, \quad \quad T_1(x,y) = T_1 ^-(x,y) T_1 ^+(x,y) = x^2y^2 -1,
\]
which cut out the curves $C_{1+}$, $C_{1-}$ and $C_1 = C_{1+} \cup C_{1-}$: 
\[
C_{1\pm} := \{T_{1\pm} = 0 \} \quad \text{and} \quad C_1 = \{T_1 = 0\}.
\]
Then $C_{1+} \cap C_{1-} = \text{\O}$.

We shall construct a function $g_{\delta} \in \cO (C_{1+})$ such that 
\begin{equation}\label{bounded-on-C+}
|g_{\delta}(\delta ^{-1}, \delta )|^{-(\delta ^{-2} + \delta ^2)} \sim 1 \quad \text{and} \quad \int _{C_{1+}}|g_{\delta}|^2 e^{-|\cdot |^2}\omega _o \le C/\sqrt{\delta}
\end{equation}
for some constant $C> 0$ independent of $\delta$.   Assuming for the moment that such a function has been found, if we define the function $f_{\delta} \in \cO (C_1)$ by 
\[
f_{\delta}(z) = \left \{ 
\begin{array}{l@{,\quad}l}
g_{\delta}(z) & z \in C_{1+} \\
0 & z \in C_{1-} 
\end{array}
\right . ,
\]
then 
\[
|f_{\delta}(\tfrac{1}{\delta}, \delta )|^{-(\tfrac{1}{\delta^2} + \delta ^2)} \sim 1, \quad |f_{\delta}(\tfrac{1}{\delta}, \delta )|^{-(\tfrac{1}{\delta^2} + \delta ^2)} = 0 \quad \text{and} \quad  \int _{C_1}|f_{\delta}|^2e^{-|\cdot|^2}\omega _o \le C/\sqrt{\delta},
\]
and in particular  $f_{\delta} \in \fB _1 (C_1)$.  To prove Theorem \ref{main-C1} by contradiction, suppse there exists $F \in \sB_2$ extending $f_{\delta}$.  Since the square norm of $f_{\delta}$ is bounded by $C/\sqrt{\delta}$,  Proposition \ref{surj-has-bdd-inverse} says one can find $F_{\delta} \in \sB _2$ such that $||F_{\delta}||^2 \le \tilde C/\sqrt{\delta}$ for some constant $\tilde C$ independent of $\delta$.   But then by \eqref{c1-bi} of Proposition \ref{w-Bergman-ineq}
\begin{eqnarray*}
\frac{1}{2\delta} &=& \frac{|F(\delta , \delta ^{-1})|^2 e^{- (\delta^2 + \delta ^{-2})} - F(-\delta , \delta ^{-1})|^2 e^{- (\delta^2 + \delta ^{-2})}}{2\delta} \\
&=& \frac{1}{2\delta} \int _{-1} ^1 \frac{d}{ds}\left (  |F_{\delta}(\delta^{-1}, s\delta) |^2e^{-\vp(\delta ^{-1} ,s\delta)} \right ) ds \\
&\le &  \sup _{\bbC ^2} \left | d(|F_{\delta}|^2e^{-\vp})\right | \le \hat C /\sqrt{\delta},
\end{eqnarray*}
where the constant $\hat C$ is independent of $\delta$.  The desired contradiction is thus obtained by taking $\delta$ sufficiently small.

\subsubsection{\bf Conclusion of the proof in the standard Bargmann-Fock space}

It remains only to produce the $g=g_{\delta}$ on $C_{1+}$ satisfying \eqref{bounded-on-C+}.  We shall define a function close to $g_{\delta}$ on a large but finite open subset of $C_{1+}$, and then approximately extend the example to all of $C_{1+}$ using H\"ormander's Theorem, thus obtaining $g_{\delta}$.

We work on $\bbC^*$, after using the parametrization 
\begin{equation}\label{nu-defn}
\nu : \bbC ^* \ni t \mapsto (t , t^{-1}) \in C_{1+}
\end{equation}
of $C_{1+}$.  Our $L^2$ norm is then 
\[
\int _{C_{1+}} |g|^2 e^{-|\cdot |^2} \omega _o = \int _{\bbC ^*} |f(t)|^2e^{-|t|^2 - |t|^{-2}} (1+|t|^{-4}) dA(t),
\]
where $f=\nu ^* g$.  Note that for the weight $\vp _o (t) := |t|^2 + |t|^{-2} - \log (1+|t|^{-4})$,  
\[
\frac{\di ^2 \vp _o }{\di t \di \bar t} =  1 + |t|^{-4} - \frac{4 |t|^2}{(1+|t|^4)^2}  = \frac{|t|^{12} +3|t|^4(|t|^2-1)^2 + 2|t|^6 +1 }{|t|^4(1+|t|^4)^2}, 
\]
which is positive, $\to 1$ as $|t| \to \infty$ and $\to \infty$ as $|t| \to 0$.  Thus 
\[
dd^c \vp _o \ge c_o  dd^c |t|^2
\]
for some positive constant $c_o$.  Moreover, there exists a compact subset $K \relcomp \bbC$ (necessarily containing the origin) such that 
\begin{equation}\label{curv-good-outside-cpct}
dd^c \vp _o \le 2 dd^c |t|^2 \quad \text{for }t \in \bbC ^* - K.
\end{equation}

Now fix $\delta \in (0,1)$, keeping in mind that we will let $\delta \to 0$.  To find a function $f$ such that $|f(1/\delta)|^2 e^{-\delta ^2}e^{ - \delta ^{-2}} \sim 1$ for some $\delta << 1$, one need only worry about the factor $e^{-\delta ^{-2}}$, which is extremely small.  A natural choice is 
\[
f_o (t) = e^{t^2/2},
\]
which satisfies $|f_o(1/\delta)|^2 e^{-\delta ^2}e^{ - \delta ^{-2}} = e^{-\delta ^2}$.  Unfortunately the function 
\[
|f_o (t)|^2 e^{-\vp _o (t)} = e^{\re t^2 - |t|^2 - |t|^{-2}} (1+|t|^{-4}) = e^{-2 (\im t)^2}e^{- |t|^{-2}} (1+|t|^{-4}),
\]
while locally integrable near the origin in $\bbC$, is not integrable in a neighborhood of $\{ \infty\}$, where it is asymptotically $e^{-2(\im t)^2}$.  Thus we are going to take the function $\chi f_o$, where $\chi$ is a cut-off function to be described shortly, and then correct this function using H\"ormander's Theorem.

Consider the vertical strip 
\begin{equation}\label{s-defn}
S_{\delta} := \left \{ t \in \bbC \ ;\ |\re t - \tfrac{1}{\delta} | \le \tfrac{3}{4\sqrt{\delta}}\right \} \subset \bbC ^*.
\end{equation}
Take a function $\chi \in \sC ^{\infty} _o (S_{\delta})$ such that 
\[
0 \le \chi \le 1, \quad \chi (t) \equiv 1 \text{ for }|\re t - \tfrac{1}{\delta} | \le \tfrac{1}{2\sqrt{\delta}} \quad \text{and} \quad \sup _{S_{\delta}}|d \chi| \le 5\sqrt{\delta}.
\]
(Such a function can be chosen to depend only on $\re t$, for instance.)  

The function $\chi f_o$ satisfies 
\[
\int _{\bbC ^*} |\chi f_o|^2 e^{-\vp_o} dA \le C\delta ^{-1/2}
\]
for some constant $C$ that does not depend on $\delta$.  Moreover, 
\[
\int _{\bbC ^*} | \dbar \chi f_o|^2 e^{-\vp_o} dA \le \sup |d\chi|^2 \int _{{\rm Supp}(\chi)} |f_o|^2 e^{-\vp_o} dA \le C\sqrt{\delta}.
\]
By H\"ormander's Theorem there exists a function $u$ such that 
\[
\dbar u = \dbar \chi f_o \quad \text{and} \quad \int _{\bbC ^*} |u|^2e^{-\vp_o} dA \le A \sqrt{\delta}.
\]
Note in particular that 
\[
u \in \cO (\{ t \ ;\ |\re t - \delta ^{-1}| < 1/(2\sqrt{\delta})\}.
\]
Moreover, if $\delta$ is small enough then by \eqref{curv-good-outside-cpct} and the proof of Proposition \ref{w-Bergman-ineq} (c.f. Remark \ref{proof-of-wbi-rmk})
\[
|u(\delta ^{-1})|^2 e^{-\delta ^2 - \delta ^{-2}} = |u(\delta ^{-1})|^2 e^{-\vp _o(\delta ^{-1})}(1+\delta ^4)^{-1} \le C\sqrt{\delta}.
\]
(This estimate is established by estimating $|u(\delta ^{-1})|^2e^{-\vp(\delta ^{-1})}$ by its $L^2$ norm over the disk of radius $1$ and center $1/\delta$ using the QuimBo Trick.)  Hence the function 
\[
f := \chi f_o - u
\]
satisfies 
\[
\int _{\bbC^*} |f|^2 e^{-\vp_o} \lesssim \delta ^{-1/2} \quad \text{and} \quad |f(\delta ^{-1})|^2 e^{-\delta ^2 - \delta ^{-2}} \sim (1 + \sqrt{\delta})e^{-\delta ^2} \sim 1 
\]
Letting 
\[
g_{\delta} (t,t^{-1}) := f(t)
\]
provides the function satisfying \eqref{bounded-on-C+}, and hence proves Theorem \ref{main-C1} in the standard case.  
\qed

\subsection{The general case}

The passage to the general case involves using the QuimBo Trick in the form of Lemma \ref{weight-normalization} to reduce to a situation that is very similar to the standard case.  In particular, we will be brief when stating estimates in this setting that are very similar to those of the standard case. 

First we normalize the weight $\vp$ in the bidisk 
\[
D^2 _{\delta} := \left \{ (x,y) \in \bbC ^2\ ;\ |x| < {2}/{\delta}\ ,\ |y| < 1 \right \}
\]
via Lemma \ref{weight-normalization}.  Thus we have functions $\psi _{\delta} \in \sC^2 (D^2_{\delta})$ and $h _{\delta} \in \cO (D^2_{\delta})$ such that 
\[
\vp = m |\cdot |^2 + \psi _{\delta} + 2 \re h_{\delta}  \quad \text{and} \quad ||\psi _{\delta}||_{\sC^1} \le C
\]
for some constant $C$ independent of $\delta$.  In particular, this relation holds on 
\[
\tilde S_{\delta} \subset D^2 _{\delta},
\]
where (compare \eqref{s-defn})
\[
\tilde S _{\delta} := \nu (S_{\delta}) =  \left \{ (t,1/t)\ ;\ |\re t - \tfrac{1}{\delta}| \le \tfrac{3}{4\sqrt{\delta}} \right \}.
\]
Again, pulling back by $\nu$, we work on $\bbC^*$, where the $L^2$ norm is 
\[
\int _{\bbC^*} |f(t)|^2e^{- \vp(t,t^{-1})} (1+|t|^{-4}) dA(t).
\]
This time, however, the weight $\vp _o (t) = \vp (t, t^{-1}) - \log (1+|t|^{-4})$ could fail to be positively curved if $m$ is sufficiently small.  We therefore need to choose a weight for which H\"ormander's Theorem can be applied, and that still provides the right estimates.  With this in mind, we let 
\[
\eta (t) := \vp (t, t^{-1})  - \frac{m}{2} |t|^{-2}.
\]
Then 
\[
dd^c \eta (t)  \ge \frac{m}{2} \nu ^* \omega _o \ge \frac{m}{2} dd^c |t|^2
\]
and 
\[
e^{-\vp(t, t^{-1})} (1+|t|^{-4}) = e^{-\eta(t)} (1+|t|^{-4})e^{-\frac{m}{2} |t|^{-2}} \le C_m e^{-\eta (t)}
\]
for some constant $C_m$.

Now let $\chi \in \sC ^{\infty}_o ([0,3/4))$ have the property that $0 \le \chi \le 1$, $\chi (r) \equiv 1$ for $0 \le r \le \frac{1}{2}$ and $|\chi '| \le 5$.  Define 
\[
\tilde f(t) := \chi (\sqrt{\delta} |\re t - \delta|) \chi (\delta |\im t|) e^{g_{\delta}(t, t^{-1})+ mt^2/2}.
\]
Then by Lemma \ref{weight-normalization}, 
\[
\int _{\bbC ^*} |\tilde f|^2 e^{-\vp_o} dA \lesssim \int _{\tilde S_{\delta}} e^{m(\re t^2- |t|^2) - \frac{m}{2} |t|^{-2}} dA(t) \lesssim \delta ^{-1/2},
\]
where the last estimate is proved as in the standard case.  Also, since 
\[
\left |\dbar \left ( \chi (\sqrt{\delta} |\re t - \delta|) \chi (\delta |\im t|)  \right ) \right |^2 \lesssim \delta \cdot \mathbf{1}_{\tilde S_{\delta}},
\]
we have 
\[
\int _{\bbC ^*} |\dbar \tilde f|^2 e^{-\eta} dA \lesssim \delta \int _{\tilde S_{\delta}} e^{m(\re t^2- |t|^2) - \frac{m}{2} |t|^{-2}} dA(t) \lesssim \delta ^{1/2}.
\]
By H\"ormander's Theorem there exists a smooth function $u$ such that 
\[
\dbar u = \dbar \tilde f \quad \text{and} \quad \int _{\bbC^*} |u|^2e^{-\vp _o} dA \lesssim \int _{\bbC^*} |u|^2e^{-\eta} dA = O(\delta^{1/2}).
\]
Since $u$ is holomorphic in a neighborhood of $\delta^{-1}$, as in the standard case we have (via Remark \ref{proof-of-wbi-rmk})
\[
|u(\delta^{-1})|^2 e^{-\vp(\delta, \delta ^{-1})} \lesssim \delta ^{1/2}.
\]
It follows that the function 
\[
f(t, t^{-1}) := \chi (\sqrt{\delta} |\re t - \delta|) \chi (\delta |\im t|) e^{g_{\delta}(t, t^{-1})+ mt^2/2} - u (t)
\]
is holomorphic and satisfies 
\[
|f(\delta ^{-1}, \delta)|^2 e^{-\vp (\delta ^{-1}, \delta)} \sim 1 \quad \text{and} \quad \int _{C_+} |f|^2e^{-\vp} \omega _o \lesssim \delta ^{-1/2}.
\]
As in the standard case, $f$ has no extension in $\sB _2 (\vp)$ as soon as $\delta$ is small enough.  The proof of Theorem \ref{main-C1} is complete.
\qed

\section{Proof of Theorem \ref{main-S}}\label{S-proof-section}

Throughout this section we make use of the `graph embedding' $\Phi : S \emb \bbC ^3$ defined by \eqref{S-graph-emb}.

\subsection{Asymptotics of the norm from the $L^2$ extension theorem}\label{ot-norm-asymp-subsection}

Consider the $L^2$ Extension Theorem \ref{ot-tak} in the following situation.  We take $X= \bbC ^3, Z= S$, $T (x,y,z) := z- xy^2$, $\vp$ as in the hypotheses of Theorem \ref{main-S}, i.e., satisfying \eqref{bf-weight-hyp}, $s = 1$ and $\lambda$ defined by 
\begin{equation}\label{lambda-S-defn}
\lambda (x,y,z) := \frac{1}{{\rm Vol}(B_R(0))} \int _{B_R(x,y,z)} \log |T(\xi, \eta, \zeta)|^2 dV(\xi,\eta,\zeta).
\end{equation}
Notice that $\lambda$ is plurisubharmonic, i.e., $dd^c \lambda \ge 0$.  As we saw in the proof of Proposition \ref{alg-zero-density}, given any $\ve > 0$, there exists $R$ sufficiently large such that 
\[
dd^c \lambda \le \ve \omega _o.
\]
Hence the curvature hypothesis \eqref{curv-ot} of Theorem \ref{ot-tak} is satisfied if we take $R>> 1$.  We fix $R$ from here on.

It follows that for any $f \in \cO (S)$ satisfying 
\begin{equation}\label{ot-subvar-norm}
\int _S \frac{|f|^2 e^{-\vp}}{|dT|^2_{\omega _o} e^{-\lambda}} \omega _o ^2 < +\infty
\end{equation}
there exists $F \in \sB _3(\vp)$ such that $F|_S = f$.  (There is also an estimate for the $\sB_3 (\vp)$-norm of $F$ in terms of the norm \eqref{ot-subvar-norm} of $f$, but this estimate is not important to us.)

Now, $dT = dz - y^2 dx - 2xy dy$, so 
\[
|dT|^2_{\omega _o} = 1 + 4|xy|^2 +|y^2|^2.
\]

\begin{lem}\label{lambda-lemma}
There exists $C > 0$ such that 
\[
 \frac{1}{C} (1+|x|^2+ |y|^2 + 4|xy|^2 + |y|^4) \le e^{\lambda(x,y,xy^2)} \le C (1+|x|^2+ |y|^2 + 4|xy|^2 + |y|^4).
\]
\end{lem}

\begin{proof}
We have 
\begin{eqnarray*}
\lambda (x,y,xy^2) &=& \frac{6}{\pi ^3 R^6} \int _{B_R(0)} \log |\zeta +xy^2 - (\xi +x)(\eta +y)^2|^2 dV(\xi, \eta, \zeta)\\
&=& \frac{6}{\pi ^3 R^6} \int _{B_R(0)} \log |\zeta - \xi \eta ^2 -(x\cdot \eta ^2 +y \cdot 2\xi \eta + xy \cdot 2\eta +y^2 \cdot \xi)|^2 dV(\xi, \eta, \zeta)\\
& \le &  \frac{6}{\pi ^3 R^6} \int _{B_R(0)}  \log \left (|\zeta| + |\xi \eta ^2| + |(\eta ^2, 2\xi \eta, 2\eta, \xi)| \cdot |(x,y,xy, y^2)|\right )^2 dV(\xi, \eta, \zeta)\\
& \le &  \frac{6}{\pi ^3 R^6} \int _{B_R(0)} \log \left (2 R^3 + 2R^2 \sqrt{|x|^2 + |y|^2 + 4|xy|^2 + |y^2|^2}\right )^2 dV(\xi, \eta, \zeta)\\
&=& \log \left ( |x|^2 + |y|^2 + 4|xy|^2 + |y^2|^2 \right ) \\
&& +  \frac{6}{\pi ^3 R^6} \int _{B_R(0)} \log \left (\frac{2 R^3}{\sqrt{|x|^2 + |y|^2 + 4|xy|^2 + |y^2|^2}} + 2R^2 \right )^2 dV(\xi,\eta,\zeta).
\end{eqnarray*}
Thus if $(x,y)$ is outside a ball of radius $R$ then the second term is small.  For $(x,y,xy^2)$ such that $(x,y)$ is inside the ball of radius $R$, $\lambda$ is already bounded. 

To obtain the opposite inequality,  one proceeds as follows.  First, note that 
\[
\lambda (x,y,xy^2) \ge \frac{6}{\pi ^3 R^6} \int _{B_R(0)} \log (|(x\cdot \eta ^2 +y \cdot 2\xi \eta + xy \cdot 2\eta +y^2 \cdot \xi)| - |\zeta - \xi \eta ^2|)^2 dV(\xi, \eta, \zeta).
\]
The region of $B_R(0)$ where 
\[
(|(x, y, 2xy, y^2) \cdot (\eta ^2 ,2\xi \eta, \eta,\xi)| - |\zeta - \xi \eta ^2|)^2 < 1
\]
is of no concern to us, since $\log |t|$ is locally integrable near $t =0$.  Thus we may focus on the region $B_+ \subset B_R(0)$ where $ (|(x, y, 2xy, y^2) \cdot (\eta ^2 ,2\xi \eta, \eta,\xi)| - |\zeta - \xi \eta ^2|)^2 \ge 1$ at the cost of subtracting some possibly large but fixed constant (depending only on $R$).  We may also assume that $|x|^2 + |y|^2 \ge R^7$.  The aforementioned region contains the set  
\[
\fC _R (x,y) := \left \{ (\xi, \eta, \zeta) \in B_R(0)\ ;\ |(x, y, 2xy, y^2) \cdot (\eta ^2 ,2\xi \eta, \eta,\xi)| \ge \tfrac{1}{2} ||(x, y, 2xy, y^2)|| \right \},
\]
and $\fC_R (x,y)$ has positive volume in $B_R(0)$ uniformly in $R$.  It follows that 
\begin{eqnarray*}
\lambda (x,y,xy^2) &\ge & M_R \int _{\fC_R (x,y)} \log ( \tfrac{1}{2} ||(x, y, 2xy, y^2)|| - |\zeta - \xi \eta ^2|)^2 dV(\xi, \eta, \zeta)\\
&=& M_R {\rm Vol} (\fC_R (x,y))\log (1+ |x|^2 + |y|^2 + 4|xy|^2 + |y|^4) \\
&& \qquad + M_R \int _{\fC(x,y)} \log \left ( \tfrac{1}{2} - \frac{ 1+ 2 |\zeta - \xi \eta ^2|}{\sqrt{1+ ||(x, y, 2xy, y^2)||^2}}\right )^2 dV(\xi, \eta, \zeta)\\
&\ge& C_o \log  (1+ |x|^2 + |y|^2 + 4|xy|^2 + |y|^4),
\end{eqnarray*}
and the proof is complete.

\end{proof}

By Lemma \ref{lambda-lemma} we have 
\[
\int _S \frac{|f|^2 e^{-\vp}}{|dT|^2_{\omega _o} e^{-\lambda}} \omega _o ^2  \leq \int _S \frac{|f|^2 e^{-\vp}}{(1+ 4|xy|^2 + |y|^4)} (1+ |x|^2 + |y|^2 + 4|xy|^2 + |y|^4) \omega _o ^2
\]
If we now use our parametrization $\Psi _2 :\bbC \ni (x,y) \mapsto (x,y,xy^2) \in S$, we find that 
\[
\Psi _2 ^* \frac{\omega _o ^2}{2} = (1 + 4|xy|^2 + |y|^4) dV(x,y),
\]
and therefore 
\begin{equation}\label{S-equiv-norms}
\int _S \frac{|f|^2 e^{-\vp}}{|dT|^2_{\omega _o} e^{-\lambda}} \omega _o ^2 \leq  2\int _{\bbC ^2} |f(x,y,xy^2)|^2 e^{-\vp (x,y,xy^2)} (1+|x|^2+ |y|^2 + 4|xy|^2 + |y|^4) dV(x,y).
\end{equation}

The norm on the right hand side of \eqref{S-equiv-norms} is larger than the $\fH_2(S,\vp)$-norm, obtained when the weight factor $(1+|x|^2+ |y|^2 + 4|xy|^2 + |y|^4)$ is replaced by the weight factor $(1+ 4|xy|^2 + |y|^4)$.  The factors become incomparable if $y=0$ and $x$ is very large.  Of course, uniform flatness fails only on a neighborhood of the line 
\[
L_o := \{ (x,y,z)\in S\ ;\ y=0\}
\]
near $\infty$.  More precisely, for any $\delta > 0$ the set 
\[
S_{\delta} := S \cap \Omega _{\delta}
\]
is uniformly flat in $\bbC ^3$, where 
\[
\Omega _{\delta} = \{ (x,y,z) \in \bbC ^3\ ;\  |y| > \delta \}.
\]
Note that $\Omega_{\delta}$, being a product of pseudoconvex domains, is pseudoconvex (and hence the $L^2$-extension theorems apply to it).
\subsection{Reduction}

\begin{prop}\label{extend-from line}
Let $\vp \in \sC^2 (\bbC ^3)$ be a Bargmann-Fock weight, i.e., a weight satisfying \eqref{bf-weight-hyp}.  Then for each $h \in \fB_1 (L_o,\vp)$ there exists $H\in \sB_3(\vp)$ such that $H|_{L_o} = h$.
\end{prop}

\begin{proof}
One simply applies Theorem \ref{interp-thm-pv-osv} twice.  Since $L_o$ is a uniformly flat hypersurface with upper density $0$ in $\bbC ^2  \cong \cP _1 = \{ (x,y,z)\in \bbC ^3\ ;\ z=0\}$, there exists $\tilde h \in \fB _2(\cP_1, \vp)$ such that $\tilde h |_{L_o} = h$. And since $\cP_1$ is a uniformly flat hypersurface with upper density $0$ in $\bbC ^3$, there exists $H \in \sB _3(\vp)$ such that $H|_{\cP_1} = \tilde h$.  Since $L_o \subset \cP_1$, $H|_{L_o} = \tilde h|_{L_o} = h$.  
\end{proof}

\begin{prop}\label{bdd-rest-S}
The restriction map $\sR_{S} : \sB_3 (\vp) \to \fB_2(S,\vp)$ is bounded.
\end{prop}

\begin{s-rmk}
Note that if $S$ were uniformly flat then Proposition \ref{bdd-rest-S} would be a consequence of a well-known fact.  But although $S$ is not uniformly flat, it is not far from being so.
\red
\end{s-rmk}

\begin{proof}[Sketch of proof of Proposition \ref{bdd-rest-S}]
Let $F \in \sB_3 (\vp)$.  Using the sub-mean value property \eqref{c0-bi} we can estimate the integral of $|F|^2$ over $S$ with respect to the measure $e^{-\vp} \omega _o ^2$ by the integral over the union of all disks of radius $\ve$ and center on $S$ that are perpendicular to $S$.  Away from some neighborhood of the line $L_o$ these disks will be mutually disjoint.  Near the line $L_o$ each point of $\bbC^3$ is in at most two such disks.  Hence the integral of $|F|^2$ over $S$ is bounded by $2 ||F||^2$ times the constant from the sub-mean value property.
\end{proof}

For any $f \in \fB_2(S,\vp)$ Proposition \ref{extend-from line} yields a function $H \in \sB _3 (\vp)$ such that $H|_{L_o} = f|_{L_o}$.  It follows that the function $g := f- H|_S \in \cO (S)$ vanishes along $L_o$, and by Proposition \ref{bdd-rest-S} $g \in \fB _2 (S,\vp)$.  Thus it suffices to extend functions $g \in \fB _2 (S,\vp)$ that vanish along $L_o$.

\subsection{Extension of functions in $\mathbf{\fB_2(S, \vp)}$ that vanish along $\mathbf{L_o}$}

\subsubsection{\bf Extension away from $\mathbf{L_o}$}

We can apply the $L^2$ Extension Theorem \ref{ot-tak} with $X= \Omega _{\delta}$, $Z= S_{\delta}$, and with $T$, $s=1$, $\vp$ and $\lambda$ as in Subsection \ref{ot-norm-asymp-subsection}.  In $S_{\delta}$ we have 
\[
|dT|^2_{\omega _o} e^{-\lambda} \ge C \frac{1}{1+\frac{|x|^2 + |y|^2}{1 + 4|xy|^2+|y|^4}} \ge C \frac{1}{1+\frac{4|x|^2 + |y|^2}{1 + \delta^2(4|x|^2+|y|^2)}} \ge \frac{C\delta ^2}{1+\delta ^2}
\]
Hence for any $f \in \fB_2 (S, \vp)$ one has 
\[
\int _{S_{\delta}} \frac{|f|^2e^{-\vp}}{|dT|^2_{\omega _o} e^{-\lambda}} \omega _o ^2 \le C_o\delta ^{-2} \int _{S_{\delta}} |f|^2e^{-\vp}\omega _o ^2,
\]
So we find a function $F_{\delta} \in \cO (\Omega _{\delta})$ such that 
\[
F_{\delta} |_{S_{\delta}} = f|_{S_{\delta}} \quad \text{and} \quad \int _{\Omega _{\delta}} |F|^2 e^{-\vp} dV < +\infty.
\]
This argument applies to all $f \in \fB _2 (S,\vp)$, and not just those $f$ that vanish along $L_o$.

\subsubsection{\bf Division by $\mathbf{y}$ in $\mathbf{\fB_2 (S,\vp)}$}  

Consider the domain 
\[
\widetilde \Omega _{\delta} := \{ (x,y,z) \in \bbC ^3\ ;\ |y| < 2\delta\}
\]
and the surface $\widetilde S_{\delta} := S \cap \widetilde \Omega _{\delta}$.

\begin{lem}\label{div-y-lem}
Fix $\delta > 0$.  If $h \in \cO (S)$ and $y h \in \fB_2 (S,\vp)$ then $h \in \fB_2(S,\vp)$.
\end{lem}

\begin{proof}
We work in $\bbC ^2$ after pulling back by the embedding $\Phi$.  Then $f \in \fB_2 (S, \vp)$ means that the $L^2$ norm 
\[
\int _{\bbC^2} |f(x,y)|^2 e^{-\vp (x,y,xy^2)} (1+ 4 |xy|^2 + |y|^4) dV(x,y)
\]
is finite.  Thus the hypothesis $yh \in \fB_2 (S,\vp)$ means that
\[
\int _{\bbC^2} |y\Phi ^* h|^2 e^{-\Phi ^* \vp} (1+ 4 |xy|^2 + |y|^4) dV(x,y) < +\infty
\]
and thus there exists $A>0$ such that for any $R>2\delta$ 
\[
\int _{D_R(0)} \int _{|y| \le 2\delta} |y\Phi ^*h(x,y)|^2 e^{-\Phi ^* \vp(x,y)}(1+4|xy|^2+|y|^4) dA(y) dA(x) \le A.
\]

By Lemma \ref{weight-normalization} (with $m=0$) there exist functions 
\[
g \in \cO (D_R(0) \times D_{2\delta}(0)\times D_{4\delta ^2 R}(0)) \quad \text{and} \quad \psi \in \sC ^2 (D_R(0) \times D_{2\delta}(0)\times D_{4\delta ^2 R}(0))
\]
such that 
\[
\vp = \psi + 2 \re g  \quad \text{and} \quad |\psi | \le C
\]
on $D_R(0) \times D_{2\delta}(0)\times D_{4\delta ^2 R}(0)$.  Note that the constant $C$ depends on $\delta$ but not on $R$.

Now, by elementary complex analysis there exists a universal constant $c_o$ so that for any entire holomorphic function $f \in \cO (\bbC ^2)$
\[
\int _{|y| \le 2\delta} |\Phi^*(he^{-g})(x,y) f(x,y)|^2 dA(y) \le c_o \int _{\delta \le |y| \le 2\delta} |\Phi^*(he^{-g})(x,y)  f(x,y)|^2 dA(y).
\]
Applying this estimate with $f(x,y)$ equal to $1$, $2xy$ and $y^2$ yields 
\begin{eqnarray*}
&& \int _{D_R(0)} \int _{|y| \le 2\delta} |\Phi ^*h(x,y)|^2 e^{-\Phi ^* \vp(x,y)}(1+4|xy|^2+|y|^4) dA(y) dA(x)\\
&\le & e^C \int _{D_R(0)} \int _{|y| \le 2\delta} |\Phi ^*(e^{-g}h)(x,y)|^2 (1+4|xy|^2+|y|^4) dA(y) dA(x)\\
&\le & c_o e^C \int _{D_R(0)} \int _{\delta \le |y| \le 2\delta} |\Phi ^*(e^{-g}h)(x,y)|^2 (1+4|xy|^2+|y|^4) dA(y) dA(x)\\
&\le & \frac{c_o e^C}{\delta^2} \int _{D_R(0)} \int _{\delta \le |y| \le 2\delta} |y \Phi ^*(e^{-g}h)(x,y)|^2 (1+4|xy|^2+|y|^4) dA(y) dA(x)\\
&\le & \frac{c_o e^{2C}}{\delta^2} \int _{D_R(0)} \int _{|y| \le 2\delta} |y \Phi ^*h(x,y)|^2e^{-\Phi ^* \vp(x,y)} (1+4|xy|^2+|y|^4) dA(y) dA(x)\\
&\le & \frac{c_o e^{2C}A}{\delta^2}. 
\end{eqnarray*}
Since $A$, $C$ and $c_o$ are independent of $R$, $yh \in \fB_2 (\widetilde S_{\delta}, \vp)$, as claimed. 
\end{proof}

\subsubsection{\bf Extension near $\mathbf{L_o}$}\label{ext-near-nonflat}

Let $g \in \fB_2 (S, \vp)$ vanish along $L_o$.  Then by the Nullstellensatz and Lemma \ref{div-y-lem} $g = y \tilde g$ for some $\tilde g \in \fB_2 (S_{\delta}, \vp)$.

We can apply the $L^2$ Extension Theorem \ref{ot-tak} in this setting as well, and obtain an extension of $f|_{\widetilde S_{\delta}}$ to $\widetilde \Omega _{\delta}$ we must show that 
\[
\int _{\widetilde S_{\delta}} \frac{|g|^2e^{-\vp}}{|dT|^2_{\omega_o}e^{-\lambda}} \omega _o ^2 < +\infty.
\]
Pulling back to $\bbC ^2$ via the parametrization $\Phi$, we need to show that 
\[
\int _{\bbC} \int _{D_{2\delta}(0)} |g (x,y,xy^2)|^2e^{-\vp(x,y,xy^2)} (1+|x|^2 + |y|^2 + 4|xy|^2 + |y|^4) dA(y) dA(x) < +\infty.
\]
But 
\begin{eqnarray*}
&& \int _{\bbC} \int _{D_{2\delta}(0)} |g (x,y,xy^2)|^2e^{-\vp(x,y,xy^2)} (1+|x|^2 + |y|^2 + 4|xy|^2 + |y|^4) dA(y) dA(x)\\
&=& \int _{\bbC} \int _{D_{2\delta}(0)} |g (x,y,xy^2)|^2e^{-\vp(x,y,xy^2)} (1+ 4|xy|^2 + |y|^4) dA(y) dA(x)\\
&& + \int _{\bbC} \int _{D_{2\delta}(0)} |\tilde g (x,y,xy^2)|^2e^{-\vp(x,y,xy^2)} (|xy|^2 + |y|^4) dA(y) dA(x).
\end{eqnarray*}
Since $g$ and (by Lemma \ref{div-y-lem}) $\tilde g$ are both in $\fB_2 (\widetilde S_{\delta},\vp)$, the last two integrals are finite.  Consequently Theorem \ref{ot-tak} yields $\tilde F_{\delta} \in \cO (\widetilde S_{\delta})$ such that 
\[
\tilde F_{\delta} |_{\widetilde S_{\delta}} = f|_{\widetilde S_{\delta}} \quad \text{and} \quad \int _{\widetilde \Omega_{\delta}} |\tilde F_{\delta}|^2 e^{-\vp} dV < +\infty.
\]

\subsubsection{\bf Patching together the two extensions $\mathbf{F_{\delta}}$ and $\mathbf{\tilde F_{\delta}}$}

Fix a function $\chi \in \sC ^{\infty}_o ([0,\infty))$ such that $\chi (t) = 1$ if $0 \le t \le \delta$, $\chi (t) = 0$ if $t \ge 2\delta$, and $|\chi '| \le 2/\delta$.  Let 
\begin{eqnarray*}
\tilde F(x,y,z) &:=& (1-\chi(|y|)) F_{\delta}(x,y,z) + \chi (|y|)\tilde F_{\delta} (x,y,z)\\
&=& F_{\delta}(x,y,z) +\chi (|y|) (\tilde F_{\delta}(x,y,z) - F_{\delta}(x,y,z)).
\end{eqnarray*}
Then $\tilde F \in \sC ^{\infty} (\bbC ^3) \cap \cO (\{ (x,y,z)\ ;\ \delta \le |y| \le 2\delta\})$, $\tilde F|_S = f$, and 
\[
\int _{\bbC ^3} |\tilde F|^2 e^{-\vp} dV < +\infty.
\]

We wish to correct $\tilde F$ by subtracting from it a function $u \in \sC ^{\infty} (\bbC ^3)$ such that 
\[
\dbar u = \dbar \tilde F, \quad u |_S \equiv 0 \quad \text{and} \quad \int _{\bbC ^3} |u|^2 e^{-\vp} dV < +\infty.
\]
Consider the weight 
\[
\psi = \vp + \log |T|^2 - \lambda,
\]
where $\lambda$ is given by \eqref{lambda-S-defn}.  Then 
\begin{equation}\label{psi-properties}
-\vp \le -\psi \quad \text{and} \quad \di \dbar \psi \geq \di \dbar \vp - \di \dbar \lambda  \ge \frac{m}{2} \omega _o
\end{equation}
for $R > 0$ sufficiently large.  Note also that on $\Omega _{\delta} \cap \widetilde \Omega _{\delta}$ the difference $\tilde F_{\delta} - F_{\delta}$ is holomorphic and vanishes on $S$.  It follows that 
\[
\tilde F_{\delta} - F_{\delta} = T \cdot h
\]
for some $h \in \cO (\Omega _{\delta} \cap \widetilde \Omega _{\delta})$ which evidently satisfies 
\begin{equation}\label{hormander-data-S}
\int _{\Omega _{\delta} \cap \widetilde \Omega _{\delta}} |\tilde F_{\delta}- F_{\delta}|^2e^{-\psi}  dV = \int _{\Omega _{\delta} \cap \widetilde \Omega _{\delta}} |h|^2e^{-\vp + \lambda}  dV.
\end{equation}
We claim that the integral \eqref{hormander-data-S} is finite.  To establish this claim, note that in $\Omega _{\delta}$ the surface $S$ is uniformly flat.  Thus by the proof of Lemma 3.2 in \cite{osv} there exist $\ve _{\delta} > 0$ and $M_{\delta}>0$ such that $U_{\ve_{\delta}}(S_{\delta} \cap \widetilde S_{\delta})$ is a tubular neighborhood of $S_{\delta} \cap \widetilde S_{\delta}$ and 
\begin{equation}\label{asymptotics-of-singularity}
\lambda (x,y,z) \ge \log |T(x,y,z)|^2 - M_{\delta}
\end{equation}
for all $(x,y,z) \in \Omega _{\delta} \cap \widetilde \Omega _{\delta}- U_{\ve _{\delta}/2}$.  Now, 
\[
\int _{\Omega _{\delta} \cap \widetilde \Omega _{\delta}} |h|^2e^{-\vp + \lambda}  dV \le \int _{\Omega _{\delta} \cap \widetilde \Omega _{\delta} - U_{\ve_{\delta}}(S_{\delta} \cap \widetilde S_{\delta})} |h|^2e^{-\vp + \lambda}  dV + \int _{U_{\ve_{\delta}}(S_{\delta} \cap \widetilde S_{\delta})} |h|^2e^{-\vp + \lambda}  dV.
\]
From \eqref{asymptotics-of-singularity} we see that 
\begin{eqnarray*}
\int _{\Omega _{\delta} \cap \widetilde \Omega _{\delta} - U_{\ve_{\delta}}(S_{\delta} \cap \widetilde S_{\delta})} |h|^2e^{-\vp + \lambda}  dV &\le&  e^{M_{\delta}} \int _{\Omega _{\delta} \cap \widetilde \Omega _{\delta}} |Th|^2e^{-\vp}  dV\\
&\le& 2 e^{M_{\delta}} \left ( \int _{\widetilde \Omega _{\delta}} |\tilde F_{\delta}|^2e^{-\vp}  dV + \int _{\Omega _{\delta}} |F_{\delta}|^2e^{-\vp}  dV \right ) < +\infty.
\end{eqnarray*}
On the other hand, once again by the method of proof of \cite[Lemma 5.3]{osv} 
\begin{eqnarray*}
\int _{U_{\ve_{\delta}}(S_{\delta} \cap \widetilde S_{\delta})} |h|^2e^{-\vp + \lambda}  dV & \lesssim & \int _{U_{\ve_{\delta}}(S_{\delta} \cap \widetilde S_{\delta})- U_{\ve_{\delta}/2}(S_{\delta} \cap \widetilde S_{\delta})} |h|^2e^{-\vp + \lambda}  dV\\
&\le & e^{M_{\delta}} \int _{U_{\ve_{\delta}}(S_{\delta} \cap \widetilde S_{\delta})- U_{\ve_{\delta}/2}(S_{\delta} \cap \widetilde S_{\delta})} |Th|^2e^{-\vp}  dV\\
&\le& 2 e^{M_{\delta}} \left ( \int _{\widetilde \Omega _{\delta}} |\tilde F_{\delta}|^2e^{-\vp}  dV + \int _{\Omega _{\delta}} |F_{\delta}|^2e^{-\vp}  dV \right ) < +\infty,
\end{eqnarray*}
which establishes the finiteness of \eqref{hormander-data-S}.

Next, observe that 
\begin{eqnarray*}
\int _{\bbC ^3} |\dbar \tilde F|^2 e^{-\psi} dV &=& \int _{\bbC ^3} |\dbar \chi (|y|)|^2|F_{\delta} - \tilde F_{\delta}|^2 e^{-\psi} dV\\
& \le & \frac{C}{\delta^2} \int _{\Omega _{\delta} \cap \widetilde \Omega_{\delta}} |\tilde F_{\delta}- F_{\delta}|^2e^{-\psi}  dV < +\infty.
\end{eqnarray*}
By H\"ormander's Theorem, \eqref{psi-properties} and the regularity of $\dbar$ there exists $u \in \sC^{\infty}(\bbC ^3)$ such that 
\[
\dbar u = \dbar \tilde F \quad \text{and} \quad \int _{\bbC^3} |u|^2e^{-\vp} dV\le \int _{\bbC^3} |u|^2e^{-\psi} dV < +\infty.
\]
The finiteness of the second integral together with the smoothness of $u$ implies that $u|_S \equiv 0$.  Thus we have a function 
\[
F := \tilde F - u \in \sB_3 (\vp)
\]
such that $F|_S= \tilde F|_S = f$.  The proof of Theorem \ref{main-S} is complete.

\section{Proof of Theorem \ref{main-C2}}

\subsection{Extension from $\mathbf{C_2}$ to $\mathbf{S}$}

Without necessarily explicitly mentioning it, we shall identify $C_2$ with $C_2 \times \{1\} = S \cap \{z=1\}$. In this section we prove the following result.

\begin{d-thm}\label{C2-to-S}
Let $\vp \in \sC ^2 (\bbC^2)$ satisfy the Bargmann-Fock curvature condition \eqref{bf-weight-hyp}, and define 
\[
\tilde \vp (x,y,z) := \vp (x,y) + \ve |z-1|^2.
\]
Then for each $f \in \fB_1 (C_2, \vp)$ there exists $g \in \fB_2 (S,\tilde \vp)$ such that $g|_{C_2} = f$.
\end{d-thm}

The approach is to apply the $L^2$ Extension Theorem \ref{ot-tak}.  However, there is some work to be done before the application of Theorem \ref{ot-tak} is possible.  Of course, we will take $X= S$, $\omega = \omega _o|_S$ and $Z= C_2$, but the rest of the data is perhaps not as obvious.

\subsubsection{\bf Step 1: (Defining section and metric for the defining line bundle)}  Let us start with the function $T \in \cO (S)$ that cuts out $C_2$.  The premise is that $C_2$ is uniformly flat in $S$, and the rationale is that this is so because $C_2 \subset \bbC^2\times \{ 1\}$ is obtained from $S \subset \bbC ^3$ by intersection with the hyperplane $\{ z=1\}$.  If we pursue this clue, we should try 
\[
T(x,y,z) := z-1.
\]
Continuing with our view of $C_2$ are the intersection of $S$ with a plane in $\bbC ^3$, in our search for the weight $\lambda$ we should exploit the idea used in Paragraph \ref{OT-subsection} in the proof of Theorem \ref{interp-thm-pv-osv}.  That is to say, we should set 
\[
\lambda_r (x,y,z) := \frac{1}{\pi r^2} \int _{D_r (z)} \log |\zeta -1|^2 dA(\zeta).
\]
Since this is the same defining data as one uses for a plane in $\bbC ^3$, one sees that there is a constant $L_r > 0$ such that 
\begin{equation}\label{unif-flat-slices}
|dT(x,y,1)|^2_{\omega _o} e^{-\lambda_r(1)} \ge L_r
\end{equation}
for all $(x,y) \in \bbC ^2$.  (This fact can of course easily be verified directly as well.)  Moreover the curvature of $\lambda_r$ is non-negative and bounded above by $\frac{1}{r^2} dd^c |z|^2$, which can be made as small as we like by taking $r$ sufficiently large but fixed.

\subsubsection{\bf Step 2: (Weight modification)} Since $dd^c \lambda_r \ge 0$, Theorem \ref{ot-tak} can be applied with the weight $\tilde \vp$ if there exists $\delta > 0$ so that
\begin{equation}\label{OT-for-S-curv-condn}
dd^c \tilde \vp + {\rm Ricci}(\omega _o|_S) \ge (1+\delta) dd^c \lambda_r
\end{equation}
everywhere on $S$.  

Since $S$ is a complex submanifold of $\bbC ^3$, the curvature of $\omega _o |_S$ can be (and in fact is) smaller that the restriction to $S$ of the curvature of $\omega _o$.  The condition \eqref{OT-for-S-curv-condn} might not be satisfied even for some weights $\vp$ satisfying the Bargmann-Fock curvature bound \eqref{bf-weight-hyp}.  (As it turns out, \eqref{OT-for-S-curv-condn} is satisfied for weights $\vp$ satisfying \eqref{bf-weight-hyp} with $m$ large enough.)  We therefore need to modify the weight $\tilde \vp$ slightly.

To see things more clearly, it is useful to parametrize the surface $S$ by the map $\Phi$ defined by \eqref{S-graph-emb}.  On the other hand, it is also useful not to attach oneself too much to this parametrization.  First, note that 
\[
\omega _o |_S = \Phi ^* \omega _o = dd^c (|s|^2 + |t|^2 + |st^2|^2).
\]
An easy computation shows that 
\[
(\omega _o|_S)^2  = (1+ |2st|^2 + |t^2|^2) dd^c |s|^2 \wedge dd^c|t|^2.
\]
Letting $\Xi (s,t) = (2st, t^2)$, we see that 
\[
{\rm Ricci}(\omega _o|_S) = \Xi ^* ( dd^c (- \log (1+|z^1|^2 + |z^2|^2)))
\]
is the pullback by $\Xi$ of the curvature of the Euclidean metric on $\cO (-1) \to \bbP _2$ in the affine chart $U_o \cong \bbC ^2 \subset \bbP_2$.  Now, 
\[
dd^c (- \log (1+|z|^2)) = \frac{(1+|z|^2)\ii dz \dot \wedge d\bar z - \ii \bar z \cdot dz \wedge z\cdot d\bar z}{(1+|z|^2)^2}\le \frac{2 \ii dz \dot \wedge d\bar z}{1+|z|^2} \le 2\omega _o|_S,
\]
where $ dz \dot \wedge d\bar z = dz^1 \wedge d\bar z ^1 + dz^2 \wedge d\bar z^2$ and $z \dot \wedge d\bar z = z \dot \wedge d\bar z+z \dot \wedge d\bar z$ and $\bar z \cdot dz = z^1 d\bar z^1 + z^2 d\bar z^2$, etc.  Therefore if $\vp$ satisfies the Bargmann-Fock curvature condition with $m > 2$ then one can already apply the $L^2$ Extension Theorem \ref{ot-tak} to the weight $\psi(x,y,z) = \vp(x,y) + m |z|^2$.  

There is, however, another modification that allows us to use the extension theorem.  Namely, we let 
\[
\psi (x,y,z) := \vp (x,y) + \ve |z-1|^2- \mu (x,y) + \log (1 + |2xy|^2 + |y^2|^2),
\]
where 
\[
\mu (x,y) = \frac{1}{\pi r^2} \int _{D_r (x)} \log (1 + |2\xi y|^2 +|y^2|^2) dA(\xi).
\]
By the sub-mean value property for subharmonic functions   
\[
\log (1+ |2xy|^2 + |y^2|^2)  \le \mu (x,y).
\]
And along the curve $C_2$ 
\begin{eqnarray*}
&& \sup _{(x,y) \in C_2} \mu (x,y)  - \log (1+ |2xy|^2 + |y^2|^2) \\
&=& \sup _{t \in \bbC ^*} \frac{1}{\pi r^2} \int _{D_r(0)} \log \frac{1+ |2(\xi +t^2)t^{-1}|^2 + |t|^{-4}}{1+ |2t|^2 + |t|^{-4}} dA(\xi).
\end{eqnarray*}
For $|t|$ large 
\[
\frac{1+ |2(\xi + t^2)t^{-1}|^2 + |t|^{-4}}{1+ |2t|^2 + |t|^{-4}}  \lesssim  r^2,
\]
while for $|t| \sim 0$ 
\[
\frac{1+ |2(\xi + t^2)t^{-1}|^2 + |t|^{-4}}{1+ |2t|^2 + |t|^{-4}}  \lesssim  1,
\]
and hence there exists $A_r$ such that 
\[
\mu (x,y) - \log (1+|2xy|^2+|y|^4) \le  A_r \quad \text{ for all }(x,y) \in C_2.
\]
By taking $r$ sufficiently large we can ensure that $dd^c \mu$ is as small as we like. 

Now, in $\bbC ^3$ we have 
\begin{eqnarray*}
dd^c \psi  &=& dd^c \vp  + \ve dd^c |z-1|^2 - dd^c \mu  + dd^c (\log (1 + |\Xi (x,y)|^2))\\
& \ge & (m-\ve) dd^c (|x|^2 + |y|^2) + \ve dd^c |z-1|^2 + dd^c (\log (1 + |\Xi (x,y)|^2)).
\end{eqnarray*}
Therefore on the surface $S$ we have 
\[
dd^c \psi + {\rm Ricci} (\omega _o|_S) - (1+\delta) dd^c \lambda_r \ge (m-\ve) dd^c (|x|^2 + |y|^2) + \ve dd^c |z-1|^2 - (1+\delta ) dd^c \lambda_r,
\]
and the right hand side is non-negative if one takes $r$ sufficiently large.  

\begin{proof}[Proof of Theorem \ref{C2-to-S}]
Let $f \in \fB_1 (C_2, \vp)$.  Then 
\[
\int _{C_2} \frac{|f|^2 e^{-\psi}}{|dT|^2_{\omega _o} e^{-\lambda_r}} \omega _o \le \frac{e^{A_r}}{L_r} \int _{C_2} |f|^2e^{-\vp} \omega _o < +\infty.
\]
By Theorem \ref{ot-tak} there exists $g \in \fB_2 (S,\psi)$ such that 
\[
g|_{C_2} = f.
\]
Since $\psi \le \tilde \vp$, $g \in \fB_2 (S,\tilde \vp)$, as desired.
\end{proof}

\subsection{Extension from $\mathbf{S}$ to $\mathbf{\bbC ^3}$}

By Theorem \ref{main-S} there exists $\tilde F \in \sB_3 (\tilde \vp)$ such that 
\[
\tilde F|_{S} = g.
\]

\subsection{Restriction from $\bbC ^3$ to $\bbC ^2 \times \{ 1\}$}

Let $F(x,y) := \tilde F (x,y,1)$.  Then 
\begin{eqnarray*}
\int _{\bbC ^2} |F(x,y)|^2 e^{-\vp(x,y)} dV(x,y) &=& \int _{\bbC ^2} |\tilde F(x,y,1)|^2 e^{-\vp(x,y)} dV(x,y)\\
&\le & \frac{\ve}{\pi} \int _{\bbC ^3} |\tilde F(x,y,z)|^2 e^{-\vp(x,y)- \ve |z-1|^2} dV(x,y,z),
\end{eqnarray*}
which shows that $F \in \sB_2 (\vp)$.  Finally, 
\[
F|_{C_2} = \tilde F |_{C_2 \times \{1\}} = g|_{C_2\times \{1\}} = f,
\]
and the proof of Theorem \ref{main-C2} is complete.
\qed

\subsection{Postscript:  the non-flat pairs of $C_2$ are too close together}

Theorem \ref{main-C2} was utterly surprising to us.  We fully expected that the curve $C_2$ would not be interpolating for any Bargmann-Fock weight on $\bbC ^2$.  Indeed, we noted that the points $(\delta ^{-2} , \pm \delta )$ were very close together in $\bbC ^2$ but quite far apart in $C_2$, so we expected to be able to find a function that is large at  $(\delta ^{-2} ,\delta )$ and vanishes at  $(\delta ^{-2} , - \delta )$.  The problem with the latter goal is that it is vague; particularly, the quantitative meaning of the word `large' is here crucial.  To show that $C_2$ is not interpolating, one would need to find, for all $\delta > 0$ sufficiently small, functions $f_{\delta} \in \fB _1(C_2, \vp)$ satisfying 
\begin{equation}\label{non-interp-needs-C2}
|f_{\delta}(\delta ^{-2} ,\delta )|^2e^{-\vp (\delta ^{-2} ,\delta )} = 1, \quad f_{\delta}(\delta ^{-2},-\delta )= 0  \quad \text{and} \quad \delta ||f_{\delta}||^2 \le r  
\end{equation}
for some $r < 1$ independent of $\delta$.  Indeed, such functions would contradict \eqref{c1-bi} of Proposition \ref{w-Bergman-ineq}.

We were able to construct functions $f_{\delta}$ satisfying the first two conditions of \eqref{non-interp-needs-C2}, but the best estimate we could find for such functions is
\[
||f_{\delta}||^2 \sim \delta ^{-2}.
\]
Unable to find functions satisfying \eqref{non-interp-needs-C2}, we reluctantly had to admit to ourselves that perhaps $C_2$ is, after all, interpolating.

\section{Proof of Theorem \ref{main-sigma}}
As mentioned in the introduction, the Proof of Theorem \ref{main-sigma} is the reflection of that of Theorem \ref{main-C2}.  

\subsection{Extension from $\mathbf{C_1}$ to $\mathbf{\Sigma}$}

\begin{d-thm}\label{C1-to-Sigma}
Let $\vp \in \sC ^2 (\bbC^3)$ satisfy the Bargmann-Fock curvature condition \eqref{bf-weight-hyp}
\begin{comment}
, and define 
\[
\tilde \vp (x,y,z) := \vp (x,y,z) + \ve |z-1|^2.
\]
\end{comment}
Then for each $f \in \fB_1 (C_2,\vp)$ there exists $g \in \fB_2 (\Sigma,\vp)$ such that $g|_{C_2} = f$.
\end{d-thm}

\begin{proof}
We use the same function $T(x,y,z)= z-1$ and $\lambda _r$ as in the proof of Theorem \ref{C2-to-S}.

For the surface $\Sigma$ we use the parametrization $\Psi (s,t) = (s,t,s^2t^2)$.  The metric induced on $\Sigma$ by the Euclidean metric is then 
\[
\omega _o |_{\Sigma} = \Psi ^* \omega _o = dd^c (|s|^2 + |t|^2 + |s^2t^2|),
\]
and 
\[
(\omega _o |_{\Sigma})^2 = (1+ |2st^2|^2 + |2ts^2|^2) dd^c |s|^2 \wedge dd^c|t|^2.
\]
Thus 
\[
{\rm Ricci}(\omega _o |_{\Sigma}) = - \Pi^* dd^c \log (1+|z^1|^2 + |z^2|^2),
\]
where $\Pi (s,t) = (2st^2, 2ts^2) = 2st (s,t)$.

To apply the $L^2$ extension theorem, we again need to rid ourselves of the negative curvature contributed by ${\rm Ricci}(\omega _o |_{\Sigma})$.  As with the surface $S$, we use a modified weight 
\[
\psi (x,y,z) := \vp (x,y,z) - \nu (x,y,z) + \log (1+4|zx^2| + 4|zy^2|),
\]
where 
\[
\nu (x,y,z) =   \frac{4!}{\pi ^2r^4}\int_{B_r(x,y)} \log (1+4|z\xi^2| +4|z\eta^2|) dV(\xi,\eta).
\]
By the sub-mean value property 
\[
\log (1+4|zx^2| + 4|zy^2|) \le \nu (x,y,z),
\]
Along the curve $C_1$
\begin{eqnarray}\label{nu-bound}
\nonumber && \sup _{(x,y,z) \in C_2 \times \{1\}} \nu (x,y,z) - \log (1+4|zx^2| + 4|zy^2|) \\
&=&  \sup _{(x,y) \in C_2} \frac{4!}{\pi ^2 r^4} \int _{B_r (0)} \log \frac{1+4 |\xi +x|^2 + 4 |\eta +y|^2}{1+4|x^2| + 4|y^2|}dV(\xi,\eta)\\
\nonumber &=&  \sup _{t \in \bbC^*} \frac{4!}{\pi ^2 r^4} \int _{B_r (0)} \log \frac{1+4 |\xi +t|^2 + 4 |\eta +t^{-1}|^2}{1+4|t|^2 + 4|t|^{-2}}dV(\xi,\eta),
\end{eqnarray}
and again the last expression is bounded, as one can check by looking near $t=0$ and $t = \infty$.  Lastly, by taking $r >> 1$ we can guarantee that $dd^c \nu$ is as small as we like.

Now, $\log (1+4|zx^2| + 4|zy^2|) = \log (1+|2yx^2|^2 + |2xy^2|^2)$ on the surface $\Sigma$, and hence 
\[
dd^c \psi + {\rm Ricci}(\omega _o|_S) = dd^c \vp  - dd^c \nu .
\]
It follows that 
\[
dd^c \psi + {\rm Ricci} (\omega _o|_S) - (1+\delta) dd^c \lambda_r \ge (m-\ve) dd^c (|x|^2 + |y|^2 + |z|^2)  - (1+\delta ) dd^c \lambda_r,
\]
for some small $\epsilon$, and the latter is positive as long as $r$ is sufficiently large, which we have assumed is the case.

Now let $f \in \fB_1 (C_1, \vp)$.  Then by \eqref{unif-flat-slices} and \eqref{nu-bound}
\[
\int _{C_2} \frac{|f|^2e^{-\psi}}{|dT|^2_{\omega _o} e^{-\lambda _r}} \omega _o \lesssim \int _{C_2} |f|^2 e^{-\vp} \omega _o.
\]
By Theorem \ref{ot-tak} there exists $g \in \fB_2 (\Sigma,\psi)$ such that 
\[
g|_{C_1} = f.
\]
Therefore, $g \in \fB_2 (\Sigma,\vp)$, as desired.
\end{proof}

\subsection{End of the proof of Theorem \ref{main-sigma}}

To achieve our contradiction, suppose $\Sigma$ is interpolating with respect to some Bargmann-Fock weight function $\vp \in \sC^2 (\bbC ^3)$. 

Let $f \in \fB_1 (C_1, \vp)$.  By Theorem \ref{C1-to-Sigma} there exists $g \in \fB_2 (\Sigma, \tilde \vp)$ such that $g|_{C_1}= f$.  Since $\tilde \vp \ge \vp$, $g \in \fB_2 (\Sigma, \vp)$. By our hypothesis there exists $\tilde F \in \sB_3(\vp)$ such that $F|_{\Sigma}=g$, and hence $\tilde F|_{C_1} = f$.  Since $\vp$ is a Bargmann-Fock weight, for each $(x,y) \in \bbC ^2$ $\vp (x,y, \cdot )$ is a Bargmann-Fock weight in $\bbC$.  Hence by \eqref{c0-bi} of Proposition \ref{w-Bergman-ineq}
\[
|\tilde F(x,y,1)|^2e^{-\vp(x,y,1)} \lesssim \int _{\bbC } |\tilde F(x,y,z)|^2e^{-\vp(x,y,z)} dA (z).
\]
Let Integration over $(x,y) \in \bbC ^2$ with respect to Lebesgue measure yields 
\[
\int _{\bbC ^2} |\tilde F(x,y,1)|^2 e^{-\vp(x,y,1)} dV(x,y) \lesssim \int _{\bbC^3 } |\tilde F(x,y,z)|^2e^{-\vp(x,y,z)} dV (x,y,z) <+\infty.
\]
Letting $F(x,y) := \tilde F(x,y,1)$, we find that $F|_{C_1} = f$ and $F \in \sB_2 (\vp(\cdot, \cdot ,1))$.  In other words, 
\[
\sR _{C_1} : \sB_2 (\vp(\cdot, \cdot ,1)) \to \fB _1(C_1, \vp(\cdot, \cdot ,1))
\]
is surjective.  Since $\vp(\cdot, \cdot ,1)$ is also a Bargmann-Fock weight, Theorem \ref{main-C1} is contradicted. 
\qed 

\subsection{Postscript: the crucial difference between $S$ and $\Sigma$}

As mentioned in the second-to-last paragraph of Subsection \ref{light-ss} of the introduction, the proof of Theorem \ref{main-sigma} by contradiction to Theorem \ref{main-C1} came to us relatively quickly by that point in our research.  Nevertheless, we wondered why we couldn't just repeat the proof of Theorem \ref{main-S} for the surface $\Sigma$ in place of $S$.  Since our proof of Theorem \ref{main-sigma} is indirect, it does not help one understand what precisely goes wrong in the proof of Theorem \ref{main-S} when $S$ is replaced by $\Sigma$.

As we see it, the difficulty is by trying to imitate the step appearing in Paragraph \ref{ext-near-nonflat}.  In the case of $S$, the uniform non-flatness is concentrated along the line $L_o$, and for this line we have a nice $L^2$ extension theorem.  But in the case of $\Sigma$ the non-flat points concentrate along the variety $\Lambda _o = \{ xy=0, z=0\}$.  The union $\{ |x| < \delta\} \cup \{ |y| < \delta\} \subset \bbC ^3$ is not pseudoconvex, and it is one of the standard exercises relating to Hartogs' Phenomenon that taking the pseudoconvex hull of this set increases its size significantly.  It is here that our method breaks down.

Of course, at the end of the day there is no way to remedy this problem.  Indeed, Theorem \ref{main-sigma}, and not its opposite, is true.

\end{document}